\RequirePackage{etoolbox}
\csdef{input@path}{{style/}{graphics/}}
\documentclass[aap,MSNbibl,seceqn,dvips]{arximspdf}

\doi{10.1214/13-AAP976} 
\volume{25}
\issue{2}
\pubyear{2015}
\firstpage{429}
\lastpage{464}
\docsubty{FLA}

\makeatletter
\newcommand{\eqref}[1]{(\ref{#1})}
\newtheorem{thmm}{Theorem}[section]

\newtheorem{lem}[thmm]{Lemma}
\newtheorem{prop}[thmm]{Proposition}
\newproclaim{defn}[thmm]{Definition}
\newproclaim{nota}[thmm]{Notation}
\newproclaim{ass}[thmm]{Assumption}
\newproclaim{asss}[thmm]{Assumptions}
\newproclaim{stass}[thmm]{Standing Assumption}
\newproclaim{rem}[thmm]{Remark}
\newproclaim{exa}[thmm]{Example}
\newcommand{\Real}{\mathbb{R}}
\newcommand{\Natural}{\mathbb{N}}
\newcommand{\prob}{\mathbb{P}}
\newcommand{\Time}{\mathbb{T}}
\newcommand{\Stop}{\mathcal{T}}
\newcommand{\qprob}{\mathbb{Q}}
\newcommand{\expec}{\mathbb{E}}

\newcommand{\ame}{\mathsf{AE}_T}

\newcommand{\F}{\mathcal{F}}
\newcommand{\G}{\mathcal{G}}
\newcommand{\Sl}{\mathcal{S}}
\newcommand{\Hcal}{\mathcal{H}}
\newcommand{\ud}{\mathrm d}
\newcommand{\sign}{\operatorname{\mathsf{sign}}}

\newcommand{\oPhi}{\overline{\Phi}}

\newcommand{\oprob}{\overline{\prob}}

\newcommand{\up}{\uparrow}

\newcommand{\oK}{{\overline{K}}}
\newcommand{\oL}{{\overline{L}}}

\newcommand{\B}{\mathcal{B}}

\newcommand{\dya}{\mathbb{D}}

\newcommand{\Shat}{\overline{S}}

\newcommand{\Xhat}{\widehat{X}}

\newcommand{\oF}{\overline{\F}}
\newcommand{\oOmega}{\overline{\Omega}}
\newcommand{\Leb}{\operatorname{\mathsf{Leb}}}

\newcommand{\X}{\mathcal{X}}
\newcommand{\bF}{\mathbf{F}}

\newcommand{\bG}{\mathbf{G}}
\newcommand{\indic}{\mathbb{I}}

\newcommand{\Var}{\operatorname{\mathsf{Var}}}

\newcommand{\oX}{\overline{X}}

\makeatother

\begin{document}
\begin{frontmatter}

\title{On the stochastic behaviour of optional processes up to random times}
\runtitle{Optional processes up to random times}

\begin{aug}
\author[A]{\fnms{Constantinos}~\snm{Kardaras}\corref{}\ead[label=e1]{k.kardaras@lse.ac.uk}\thanksref{T1}}
\runauthor{C. Kardaras}
\affiliation{London School of Economics and Political Science}
\thankstext{T1}{Supported in part by NSF Grant DMS-09-08461.}
\address[A]{Statistics Department\\
London School of Economics and Political Science\\
10 Houghton Street\\
London, WC2A 2AE\\
United Kingdom\\
\printead{e1}} 
\end{aug}

\received{\smonth{3} \syear{2012}}
\revised{\smonth{10} \syear{2013}}

%
\begin{abstract}
In this paper, a study of random times on filtered probability spaces
is undertaken. The main message is that, as long as distributional
properties of optional processes up to the random time are involved,
there is no loss of generality in assuming that the random time is
actually a randomised stopping time. This perspective has advantages in
both the theoretical and practical study of optional processes up to
random times. Applications are given to financial mathematics, as well
as to the study of the stochastic behaviour of Brownian motion with
drift up to its time of overall maximum as well as up to last-passage
times over finite intervals. Furthermore, a novel proof of the
Jeulin--Yor decomposition formula via Girsanov's theorem is provided.
\end{abstract}

%
\end{frontmatter}

\section*{Introduction}\label{sec1}
Consider a filtered measurable space $(\Omega, \bF)$, where $\bF=
(\F_t)_{t
\in{\Real_+}}$ is a right-continuous filtration, as well as an underlying
sigma-algebra $\F$ over $\Omega$ such that $ \F\supseteq\F_{\infty}:= \bigvee_{t \in\Real_+} \F_t$, where the last set-inclusion may be
strict. A \textsl{random time} is a $[0,\infty]$-valued, $\F$-measurable
random variable. The interplay between random times and the filtration
$\bF$ goes a long way back, with the pioneering work of \cite
{MR509204,MR511775,MR519996}; see also the volume
\cite{MR604176}. Interest in random times has been significant, especially
in connection with applications in financial mathematics, such as
reduced-form credit risk modelling; see, for example, \cite{MR1802597,MR1722700} and~\cite{MR2811019}.

A common approach to \emph{constructing} random times is via the use of
randomised stopping times (also called Cox's method; see \cite{Lando}).
Let $\qprob$ be a probability on $(\Omega, \F)$, and suppose that there
exists an $\F$-measurable random variable $U$ that is stochastically
independent of $\F_{\infty}$ and has the standard uniform law under
$\qprob$. For a given $\bF$-adapted, right-continuous and
nondecreasing process $K = (K_t)_{t \in{\Real_+}}$ such that $0 \leq K
\leq
1$, define the random time $\psi:= \inf\{t \in{\Real_+}| K_t \geq
U\}$, where by convention we set $\psi= \infty$ if the last set
is empty. For such a duple $(\psi, \qprob)$, we say that $\psi$
is a
\emph{randomised stopping time} on $(\Omega, \F, \bF, \qprob)$.
Randomised stopping
times have several noteworthy properties:

\begin{itemize}
\item The independence of $U$ and $\F_{\infty}$ under $\qprob$ implies
that $\qprob[\psi> t | \F_t] = 1 - K_t$, for all $t \in{\Real_+}$.
Therefore, $1 - K$ represents the conditional survival process
associated to $\psi$ under \emph{any} probability $\qprob$ which makes
$U$ and $\F_{\infty}$ independent. The latter fact is useful in
modelling, for example, since $\qprob[\psi\leq t] = \expec_\qprob[K_t]$
holds for $t \in{\Real_+}$, $\qprob$ can be chosen in order to
control the
unconditional distribution of~$\psi$, while keeping the conditional
survival probabilities fixed.
\item Although $\psi$ is not a stopping time on $(\Omega, \bF)$, it
is in some
sense very close to being one. Indeed, $\psi$ is a stopping time of an
initially enlarged filtration, defined as the right-continuous
augmentation of $(\F_t \vee\sigma(U))_{t \in{\Real_+}}$. Importantly,
due to the independence of $U$ and $\F_{\infty}$ under $\qprob$, each
martingale on $(\Omega, \bF, \qprob)$ is also a martingale on the
space with the
enlarged filtration---in other words, the immersion property (\cite
{MR1648165}, also called hypothesis $(\Hcal)$ in \cite{MR511775})
holds. This opens the door to major theoretical analysis of such random
times using tools of martingale theory.
\item From a more practical viewpoint, it is straightforward to
simulate processes up to time $\psi$ under $\qprob$. One first
simulates a uniform random variable $U$; then, in an independent
fashion, one continues with simulating the process $K$ until the point
in time that it exceeds $U$, along with other processes of interest.
\end{itemize}

In view of the usefulness of randomised stopping times, it is natural
to explore their generality. Of course, it is not possible that an
arbitrary random time is a randomised stopping time, since for the
latter there is a need for the extra ``randomisation'' coming from the
uniform random variable. There is a further, more fundamental reason
that an arbitrary random time cannot be realised as a randomised
stopping time. Typically, for a random time $\rho$ on a filtered
probability space $(\Omega, \F, \bF, \prob)$, the nonnegative
process $\Real_+ \ni t
\mapsto\prob[\rho> t | \F_t]$ fails to be nonincreasing, which
would have to be the case if $\rho$ was a randomised stopping time on
$(\Omega, \bF, \prob)$. Nevertheless, the main message of the paper
is the following:

\begin{quote}
With a given a pair $(\rho, \prob)$ of a random time $\rho$ and a
probability $\prob$ on $(\Omega, \F, \bF)$, one can \emph
{essentially} associate
a pair $(\psi, \qprob)$, where $\qprob$ is a probability on $(\Omega, \F
)$ and $\psi$ is a randomised stopping time on $(\Omega, \F, \bF,
\qprob)$, such that
for \emph{any} $\bF$-optional process $Y$, the law of $(Y_{\rho
\wedge
t})_{t \in{\Real_+}}$ under $\prob$ is identical to the law of
$(Y_{\psi
\wedge t})_{t \in{\Real_+}}$ under $\qprob$.
\end{quote}
Therefore, as long as distributional properties of optional processes
on $(\Omega, \bF)$ under $\prob$ \emph{up} to the random time $\rho
$ are
concerned, there is absolutely no loss of information in passing from
$(\rho, \prob)$ to the more workable pair $(\psi, \qprob)$.

There is a reason for the qualifying ``essentially'' in the claim that
the above association can be carried out. To begin with, $\F$ should be
large enough to support a random variable $U$ that will be independent
of $\F_{\infty}$ under $\qprob$. This is hardly a concern; if the
original filtered space $(\Omega, \F, \bF)$ is not rich enough, one
can always
enlarge it in a minimal way, without affecting the structure of $\bF$,
in order to make the previous happen. However, there is another, more
technical obstacle. As will be argued in Section~\ref{sec: canonical
pair} of the text, what is guaranteed is the existence of a
nonnegative local martingale $L$ on $(\Omega, \bF, \prob)$ with
$L_0 = 1$ that is a
candidate for a local (through a specific localising sequence of
stopping times) density process of $\qprob$ with respect to $\prob$.
Then an argument ensuring that a consistent family of probabilities
constructed in ever-increasing sigma-algebras has a countably additive
extension to the limiting sigma-algebra is needed. Such an issue has
appeared in different contexts in stochastic analysis; see \cite
{MR0309184,MR0368131,MR1339738}. Under appropriate
topological assumptions on the underlying filtrations, for example,
working on canonical path-spaces as discussed in \cite{MR2169627}; one
can successfully construct a probability $\qprob$ out of $L$.

The aforementioned purely technical issue notwithstanding, the
usefulness of the above philosophy is evident. In fact, as will be made
clear in the text, even if the probability $\qprob$ cannot be
constructed, the representation pair consisting of the process $K$ in
the definition of $\psi$ and the local martingale $L$ on $(\Omega,
\bF, \prob)$
encodes significant information regarding the structure of random times.

In order to carry out the above-described program in practice, given a
random time $\rho$ on $(\Omega, \F, \bF, \prob)$ one needs to
identify the pair $(K,
L)$ associated with $\rho$. There are indeed formulas in the paper that
provide $(K, L)$ in terms of the process $\Real_+ \ni t \mapsto\prob
[\rho> t | \F_t]$ of conditional survival probabilities of $\rho$,
as well as the optional compensator on $(\Omega, \bF, \prob)$ of the
nondecreasing
process $\Real_+ \ni t \mapsto\indic_{\{\rho\leq t\}}$. Closed-form
expressions for the previous quantities are sometimes available---this
is, for example, the case when times of maximum and last-passage times
for certain nonnegative local martingales are considered. In order to
illustrate the theoretical results, applications are presented in the
context of financial mathematics, and discussion is provided regarding
times of maximum and last-passage times for finite time-horizon
Brownian motion with drift.

The dominant approach toward the study random times in the literature
is that of \emph{progressive enlargement of filtrations}. This theory
has produced remarkable results, one of the most important due to
Jeulin and Yor \cite{MR519998}, providing the canonical representation
of semimartingales up to random times under progressive enlargement of
filtrations. This result is revisited in the text, where a novel proof
of the Jeulin--Yor decomposition formula via the use of Girsanov's
theorem---a certainly more familiar result---facilitates understanding
by shedding an extra intuitive light.

\subsection*{Structure of the paper}
This introductory part ends with general remarks that will be used
throughout the text. In Section~\ref{sec: canonical pair}, the
canonical pair of processes associated with a random time is
introduced, and certain of its properties are explored in Section~\ref
{sec: canonical pair}. Section~\ref{sec: equal in law} deals with a
rigorous statement of the main message of the paper, regarding the law
of optional processes up to random times. Section~\ref{sec: examples}
contains some first examples. Section~\ref{sec: finance} presents
applications of the theory in financial settings. Section~\ref{sec: BM
on finite interv} contains a discussion on the stochastic behaviour of
Brownian motion with drift over finite time-intervals until its time of
maximum and until last-passage times. Finally, in Section~\ref{sec:
jeulin-yor} the statement and a new proof of the Jeulin--Yor
decomposition formula is provided.

\subsection*{General probabilistic remarks}
The underlying filtration $\bF= (\F_t)_{t \in{\Real_+}}$ is assumed
to be
right-continuous, but it will \emph{not} be assumed that each $\F_t$,
$t \in{\Real_+}$, is completed with $\prob$-null sets of $\F$.
Although this
relaxation calls for some technicalities, it is essential in the
development; indeed, the need for defining a probability on $(\Omega,
\F)$ that is not absolutely continuous with respect to $\prob$ (not
even locally, on each $\F_t$, $t \in\Real_+$) will arise. An extensive
part of the general theory of stochastic processes can be developed
without the completeness assumption on filtrations, as long as
properties that hold ``everywhere'' are asked to hold in an ``almost
everywhere'' sense. (Of course, there are exceptions to the previous
rule, e.g., the so-called debut theorem fails if the filtration
is not completed; see the discussion in \cite{RW-1}, Chapter II, Section 75.) The
interested reader can refer to \cite{MR1943877}, Chapters I and~II, for results in this slightly nonconventional framework
that shall be used throughout the paper. Versions of the section
theorem from \cite{MR1219534}, IV Section~1, where again the filtration
is not assumed to be completed, will also be useful.

For a c\`adl\`ag process $X$, define the process $X_- = (X_{t-})_{t
\in
{\Real_+}}$, where $X_{t-}$ is the left-limit of $X$ at $t \in(0,
\infty)$;
by convention, $X_{0-} = 0$. Also, $\Delta X:= X - X_-$. Every
predictable process $H$ is supposed to satisfy $H_0 = 0$. For any
$[0,\infty]
$-valued, $\F$-measurable random variable $\rho$ and any process $X$,
$X^\rho= X_{\rho\wedge\cdot}$ is defined as usual to be the process
$X$ stopped at $\rho$. For any c\`adl\`ag process $X$, we set $X^\up:= \sup_{t \in[0, \cdot]} X_t$, as well as $X^* = \sup_{t \in[0,
\cdot]} |X_t| = (|X|)^\up$.\vspace*{1pt}

Whenever $H$ and $X$ are processes such that $X$ is a semimartingale to
be used as an integrator and $H$ can be used as integrand with respect
to $X$, we use $\int_{[0, \cdot]} H_t \,\ud X_t$ to denote the integral process.
For a detailed account of stochastic integration, see~\cite{MR1943877}.

If not stated otherwise, a property of a stochastic process (such as
nonnegativity, path right-continuity, etc.) is assumed to hold \emph
{everywhere}; we make explicit note if these properties hold almost
surely with respect to some probability on $(\Omega, \F)$.
When processes that are (local) martingales, super-martingales, etc.,
are considered, it is tacitly assumed that their paths are almost
surely c\`adl\`ag with respect to the probability under consideration;
for example, local martingales on $(\Omega, \bF, \prob)$ have $\prob
$-a.s. c\`adl\`ag paths.

In this paper, we \emph{always} work under the following.

\begin{stass} \label{stass: finiteness}
All random times $\rho$ are assumed to satisfy $\prob[\rho< \infty]
= 1$.
\end{stass}

The reason for the above assumption is purely conventional; under its
force, $t = \infty$ does not appear explicitly in the time-indices
involved, something that would be unusual and create unnecessary
confusion. We stress, however, that Assumption~\ref{stass: finiteness}
in practice does not entail any loss of generality whatsoever. Indeed,
a~simple deterministic time-change of $[0, \infty]$ to $[0,1]$ on the
time-indices of filtrations, processes, etc., makes any $[0,\infty]$-valued
random time actually bounded; then all the results of the paper apply.

\section{A canonical pair associated with a random time} \label{sec:
canonical pair}

We keep all notation and remarks that appeared in the \hyperref[sec1]{Introduction}. In particular, Assumption~\ref{stass: finiteness} will always
be tacitly in force.

\subsection{Construction of the canonical pair} \label{subsec: canon pair}
The following result is the point of our departure.

\begin{thmm} \label{thmm: doleans sharpened}
Let $\rho$ be a random time on $(\Omega, \F, \bF, \prob)$. Then
there exists a pair of
processes $(K, L)$ with the following properties:
\begin{longlist}[(1)]
\item[(1)]$K$ is $\bF$-adapted, right-continuous, nondecreasing, with $0
\leq K \leq1$. 
\item[(2)]$L$ is a nonnegative process with $L_0 = 1$ that is a local
martingale on $(\Omega, \bF, \prob)$.
\item[(3)] For any nonnegative optional processes $V$ on $(\Omega, \bF)$, it
holds that
\[
\expec_\prob[V_\rho] = \expec_\prob \biggl[\int
_{{\Real_+}} V_t L_t \,\ud K_t
\biggr].
\]
\item[(4)]
$\int_{{\Real_+}}\indic_{\{K_{t-} = 1\}} \,\ud L_t = 0$ and $\int_{{\Real_+}}
\indic
_{\{L_t = 0\}} \,\ud K_t = 0$ hold $\prob$-a.s.
\end{longlist}
Furthermore, a pair $(L, K)$ that satisfies the above requirements is
essentially unique, in the following sense: if $(K', L')$ is another
pair that satisfies the above requirements, then $K$ is $\prob
$-indistinguishable from $K'$, while $\prob[L_t = L'_t, \forall t
\in{\Real_+}| K_{\infty} > 0] = 1$.
\end{thmm}

\begin{defn}
For a random time $\rho$ on $(\Omega, \F, \bF, \prob)$, the pair
$(K, L)$ that
satisfies requirements (1), (2), (3) and (4) of Theorem~\ref{thmm:
doleans sharpened} will be called \textsl{the canonical pair associated
with $\rho$}.
\end{defn}

\begin{rem}
Let $\rho$ be a random time on $(\Omega, \F, \bF, \prob)$ with
associated pair $(K,
L)$. Then $\rho$ is a stopping time on $(\Omega, \bF)$ if and only
if $K =
\indic_{[\![ \rho, \infty[\![}$ (and, in this case, $L \equiv1$ will
hold). Indeed, if $\rho$ is a stopping time, $K':= \indic_{[\![ \rho,
\infty[\![}$ is $\bF$-adapted, nonnegative and nondecreasing, and
$0 \leq K' \leq1$ holds. Furthermore, $\expec_\prob[V_\rho] =
\expec_\prob
[\int_{{\Real_+}} V_t \,\ud K'_t]$ holds for all nonnegative and optional
$V$ on
$(\Omega, \bF)$. By the essential uniqueness under $\prob$ of the canonical
pair associated with $\rho$, we obtain $K = \indic_{[\![ \rho, \infty
[\![}$ (and $L = 1$). Conversely, assume that $K = \indic_{[\![ \rho,
\infty[\![}$; as $K$ is $\bF$-adapted, $\rho$ is a stopping time.
\end{rem}

Given a random time $\rho$ on $(\Omega, \F, \bF, \prob)$, it will
now be explained how
the associated canonical pair $(K, L)$ is constructed. We follow the
proof of \cite{Kar09}, Theorem~2.1, which contains Theorem~\ref{thmm:
doleans sharpened} as a special case. Only details which will be
essential in the present development are provided. We also introduce
some further notation to be used throughout.

Let $Z$ be the nonnegative c\`adl\`ag super-martingale on $(\Omega,
\bF, \prob)$
that satisfies $Z_t =\prob[\rho> t | \F_t]$ for all $t \in
\Real
_+$. (The fact that such a $\prob$-a.s. c\`adl\`ag version $Z$ exists
follows from the right-continuity of the filtration $\bF$ and the
right-continuity of the function $\Real_+ \ni t \mapsto\prob[\rho>
t]\in[0,1]$ by an application of \cite{MR1219534}, Theorem~II.2.44.)
In view of Assumption~\ref{stass: finiteness}, $Z_\infty:= \lim_{t
\to\infty}
Z_t$ is $\prob$-a.s. equal to zero. Note that $Z$ is the conditional
survival process associated to a random time by Az\'ema; see \cite
{MR604176} and the references therein. Also, let $A$ be the dual
optional projection of $\indic_{[\![ \rho, \infty[\![}$ on $(\Omega,
\bF, \prob)$;
in other words, $A$ is the unique (up to $\prob$-evanescence) $\bF
$-adapted, c\`adl\`ag, nonnegative and nondecreasing process such that
$\expec_\prob[V_\rho] = \expec_\prob[\int_{{\Real_+}} V_t \,\ud
A_t]$ holds
for all
nonnegative optional process $V$ on $(\Omega, \bF)$. Then $N:= Z +
A$ is a
nonnegative martingale on $(\Omega, \bF, \prob)$ with $N_t = \expec
_\prob[A_\infty| \F_t]$, for all $t \in{\Real_+}$.

\begin{rem} \label{rem: good version of A}
Since we do not assume that the $\F_0$ contains all $\prob$-null sets
of $\F$, the properties of $A$ being c\`adl\`ag, nondecreasing and
nonnegative only are valid for $\prob$-a.s. every path. However, one
can alter $A$ to have them holding identically. Indeed, with $\dya$
denoting a countable and dense subset of $\Real_+$, define $A':=
\lim_{\dya\ni t \downarrow\cdot}  ( \sup_{s \in[0, t] \cap
\dya
} (\max\{A_s, 0 \})  )$. It is easily seen that this new
process $A'$ is $\bF$-adapted (the right-continuity of $\bF$ is
essential here), c\`adl\`ag, nondecreasing and nonnegative, and that $A
= A'$ up to $\prob$-evanescence. It is possible that $A$ can explode to
$\infty$ in finite time, but this happens on a set of zero (outer)
$\prob$-measure and will not affect the results that follow in any way.
Therefore, we might, and shall, assume in the sequel that $A$ is
c\`adl\`ag, nondecreasing and nonnegative everywhere.
\end{rem}

\begin{rem}
The expected total mass of $A$ over ${\Real_+}$ under $\prob$ is
$\expec_\prob
[A_\infty] = 1$. If $\prob[A_\infty> 1] = 0$, in which case $\prob
[A_\infty= 1] = 1$, defining $K:= A$ (more precisely, $K:= \min
\{A, 1\}$) and $L:= 1$ would suffice for the purposes of Theorem~\ref
{thmm: doleans sharpened}. However, in all other cases of random times
we have $\prob[A_\infty> 1] > 0$, and the pair $(K, L)$ is constructed
from $(A, Z)$ as will be shown below.
\end{rem}

We continue with providing the definition of the pair $(K, L)$.
Consider the stopping time $\zeta_0:= \inf\{t \in{\Real_+}| Z_{t-}
= 0 \mbox{ or } Z_t = 0\}$; in fact, $\zeta_0$ actually is the
terminal time of movement for both $Z$ and $A$. The process $K$ is
defined via
%
\begin{eqnarray}
\label{eq: real defn of K} K &=& 1 - \prob[\rho> 0] \exp \biggl(- \int_{(0, \zeta_0 \wedge\cdot]}
\frac{\ud A_t}{Z_t + \Delta A_t} \biggr)
\nonumber
\\[-8pt]
\\[-8pt]
\nonumber
&&{}\times\prod_{t \in(0, \zeta_0 \wedge
\cdot
]} \biggl(
\biggl(1 - \frac{\Delta A_t}{Z_t + \Delta A_t} \biggr) \exp \biggl(\frac{\Delta
A_t}{Z_t + \Delta A_t} \biggr)
\biggr),
\end{eqnarray}
where by convention the product of an empty set of numbers is equal to
one. It is clear that $K$ is $\bF$-adapted, c\`adl\`ag, nondecreasing and
$[0,1]$-valued on $[\![ 0, \zeta_0 [\![$. A~little care has to be
exercised in the value of $K$ at $\zeta_0$. On $\{\Delta A_{\zeta_0}
= 0\}$, it simply holds that $K_{\zeta_0} = K_{\zeta_0 -}$. On $\{
\Delta A_{\zeta_0} > 0\}$ it holds that $K_{\zeta_0} = 1$ because the
product term on the right-hand side of equation \eqref{eq: real defn of
K} is zero. The process $K$ remains constant after $\zeta_0$. In order
to get some intuition on the definition of $K$, note that the
differential equation that the process $K$ defined in \eqref{eq: real
defn of K} satisfies is
%
\begin{equation}
\label{eq: defn of K} \frac{\ud K_t}{1 - K_{t-}} =
\frac{\ud A_t}{Z_t + \Delta A_t} \qquad\mbox{for } t \in[0,
\zeta_0).
\end{equation}
For fixed $t \in[0, \zeta_0)$, $Z_t + \Delta A_t = \prob[\rho\geq t
| \F_t]$ represents the expected total remaining ``life'' of $\rho$
on $[t, \infty]$, conditioned on $\F_t$; then, formally, $\ud A_t /
(Z_t + \Delta A_t)$ is the ``fraction of remaining life of $\rho$ spent
at $t$.'' The equivalent ``fraction of remaining life spent at $t$''
for $K$ would be $\ud K_t / (1 - K_{t-})$. (The previous quantity is
based on the understanding that $\prob[K_\infty= 1] = 1$,
although this is not always the case as will be shown later in Remark~\ref{rem: K less than one, L strict loc mart }.) To get a feeling of
how $L$ should be defined, observe that $(Z + \Delta A) \Delta K =
(1 - K_-) \Delta A$ implies that $(Z + \Delta A) (1 - K)= (1 - K_-) Z$.
Therefore, from \eqref{eq: defn of K} we obtain that $\ud K_t / (1 -
K_t) = \ud A_t / Z_t$ holds for $t \in[0, \zeta_0)$, which implies
that $Z_t \,\ud K_t = (1 - K_{t}) \,\ud A_t$ holds for $t \in{\Real_+}$. Since
$\ud A_t = L_t \,\ud K_t$ has to hold for $t \in{\Real_+}$ in view of property
(3) in Theorem~\ref{thmm: doleans sharpened}, we obtain $L (1 - K) =
Z$. Using the previous equality and It\^o's formula, we obtain the dynamics
%
\begin{equation}
\label{eq: defn of L} \frac{\ud L_t}{L_{t-}} = \frac{\ud N_t}{Z_{t-}},\qquad t \in[0,
\zeta_0],
\end{equation}
where recall that $N = Z + A$. Equation \eqref{eq: defn of L} can
actually be used as the definition of $L$, which becomes equal to the
stochastic exponential of the local martingale $\int_0^{\zeta_0
\wedge
\cdot} (1/Z_{t-}) \,\ud N_t$. (One has to be quite careful here: the
latter process might not be defined at time $\zeta_0$ and onward due to
explosion, which will imply that, $\prob$-a.s., $L_t = 0$ for all $t
\geq\zeta_0$. The treatment in \cite{Kar09}, Section~2.3, makes sure
that all such issues are dealt with.) Then the relationship $Z = L(1 -
K)$ can be shown to hold true. One can check \cite{Kar09}, Section~2.3,
for all the remaining technical details of the proof.

\begin{rem} \label{rem: L as mult decomp}
When $\Delta K$ is $\prob$-evanescent (which happens exactly when
$\Delta A$ is $\prob$-evanescent), the formula $Z = L(1 - K)$ implies
that $L$ coincides with the local martingale on $(\Omega, \bF, \prob
)$ that appears
in the multiplicative Doob--Meyer decomposition of the nonnegative
$(\Omega, \bF, \prob)$-super-martingale $Z$. This fact provides the
most efficient
way to calculate the canonical pair associated with a random time that
avoids all stopping times. (For the definition and characterisation of
random times avoiding all stopping times, see Section~\ref{subsec:
avoid_stop}.)
\end{rem}

\subsection{A consistent family of probabilities associated with a
random time} \label{subsec: consist fam prob assoc with rand}

Let $\rho$ be a random time on $(\Omega, \F, \bF, \prob)$ with
associated canonical
pair $(K, L)$. Define
%
\begin{equation}
\label{eq: good stop times} \eta_u:= \inf\{t \in{\Real_+}| K_t
\geq u\}\qquad \mbox{for } u \in[0, 1),
\end{equation}
with the convention $\eta_u = \infty$ if the last set is empty. The
nondecreasing family $(\eta_u)_{u \in[0, 1)}$ of stopping times on
$(\Omega, \bF)$ will play a major role in the development. We start
with a
``localisation'' result.

\begin{lem} \label{lem: loc bdd of L}
Let $\rho$ be a random time on $(\Omega, \F, \bF, \prob)$ with
canonical pair $(K,
L)$. For $u \in[0, 1)$, $\prob[L_{\eta_u}^* \leq2 / (1 - u)] = 1$
holds. If $\prob[\eta_u < \infty, \Delta L_{\eta_u} > 0] = 0$, then
$\prob[L_{\eta_u}^* \leq1 / (1 - u)] = 1$.
\end{lem}

\begin{pf}
Fix $u \in[0, 1)$. Since $K_{t-} \leq u$ holds for $t \in[0, \eta_u]$
and $Z_- \leq1$ holds up to $\prob$-evanescence, it follows that
\[
L_- = \frac{Z_-}{1 - K_-} \leq\frac{1}{1 - u} \qquad\mbox{holds } \prob \mbox{-a.s.
on } [\![ 0, \eta_u ] \!],
\]
which implies that $\prob[L^*_{\eta_u -} \leq1/(1 - u)] = 1$. It
remains to check what happens at~$\eta_u$. In case $\prob[\eta_u <
\infty, \Delta L_{\eta_u} > 0] = 0$, $\prob[L_{\eta_u}^* \leq1 /
(1 - u)] = 1$ is immediate. Now, remove the assumption $\prob[\eta
_u < \infty, \Delta L_{\eta_u} > 0] = 0$. We shall use that
$\Delta
A \leq1$ up to $\prob$-evanescence. (Indeed, the equality $\Delta
A_{\tau} = \prob[\rho= \tau| \F_{\tau}]$ holds $\prob$-a.s. on
$\{\tau< \infty\}$ for any stopping time $\tau$, since $A$ is the
dual optional projection of $\indic_{[\![ \rho, \infty[\![}$ on
$(\Omega, \bF, \prob)
$. It follows that $\prob[\Delta A_\tau\leq1] = 1$ for any stopping
time $\tau$ and, therefore, that $\Delta A \leq1$ up to $\prob
$-evanescence by \cite{MR1219534}, Theorem~4.10.) Using \eqref{eq: defn
of L}, we obtain, $\prob$-a.s.,
\[
L_{\eta_u} = L_{\eta_u -} + \frac{\Delta N_{\eta_u}}{1 - K_{\eta_u
- }} = \frac{Z_{\eta_u -} + \Delta N_{\eta_u}}{1 - K_{\eta_u - }} =
\frac
{Z_{\eta_u} + \Delta A_{\eta_u}}{1 - K_{\eta_u - }} \leq\frac{2}{1
- u},
\]
which completes the proof.
\end{pf}

In view of Lemma~\ref{lem: loc bdd of L}, for any $u \in[0, 1)$ one can
construct a probability measure ${\mathbb{Q}_u}$ on $(\Omega, \F)$
via the
recipe $\ud{\mathbb{Q}_u}= L_{\eta_u} \,\ud\prob$. The collection
$({\mathbb{Q}_u}
)_{u \in[0, 1)}$ has the following consistency property: $\qprob_u =
\qprob_v$ on $(\Omega, \F_{\eta_u})$ holds whenever $0 \leq u \leq
v <
1$. It would be very convenient (but not {a priori} clear and
certainly not true in general, as is demonstrated in Example~\ref{exa:
need_for_no_compl}) if one could find a probability $\qprob\equiv
\qprob_1$ on $(\Omega, \F)$ such that $\qprob|_{\F_{\eta_u}} =
{\mathbb{Q}_u}|_{\F_{\eta_u}}$ holds for all $u \in[0, 1)$. This is
indeed the
case in a number of examples, as will be discussed later. The
consequences of the existence of such probability are analysed in
Section~\ref{sec: equal in law}. For the time being, we mention an
auxiliary result.

\begin{lem} \label{lem: K under Q}
For all $u \in[0, 1)$, it holds that ${\mathbb{Q}_u}[L_{\eta_u} > 0]
= 1$ and
${\mathbb{Q}_u}[\eta_u < \infty] = 1$.
\end{lem}

\begin{pf}
Fix $u \in[0, 1)$. Then ${\mathbb{Q}_u}[L_{\eta_u} > 0] = \expec
_\prob[L_{\eta_u}
\indic_{\{L_{\eta_u} > 0\}}] = \expec_\prob[L_{\eta_u}] = 1$. In
order to
show the equality ${\mathbb{Q}_u}[\eta_u < \infty] = 1$, first observe
that since $0 = Z_\infty= L_\infty(1 - K_\infty)$ holds $\prob$-a.s.,
we have $\prob[K_\infty< 1, L_\infty> 0] = 0$. Coupled with
the fact that $\{\eta_u = \infty\} \subseteq\{K_\infty< 1\}$, we
obtain $\prob[L_{\eta_u} \indic_{\{\eta_u < \infty\}}= L_{\eta
_u}] = 1$. Therefore, ${\mathbb{Q}_u}[\eta_u < \infty] = \expec
_\prob[L_{\eta_u} \indic_{\{\eta_u < \infty\}}] = \expec_\prob
[L_{\eta_u} ]
= 1$.
\end{pf}

\subsection{Time changes}
For a nonnegative $(\Omega, \bF)$-optional process $V$, the
change-of-variables formula gives $\int_{{\Real_+}} V_t \,\ud K_t =
\int_{[0, 1)}
V_{\eta_u} \indic_{\{\eta_u < \infty\}} \,\ud K_{\eta_u}$. For $a
\in
[0, 1)$, on the event $\{K_{\eta_a-} < K_{\eta_a}\}$ it holds that
\[
V_{\eta_a} \Delta K_{\eta_a} = V_{\eta_a}
(K_{\eta_a} - K_{\eta_a
- } ) = \int_{K_{\eta_a - }}^{K_{\eta_a}}
V_{\eta_a} \,\ud u = \int_{K_{\eta_a
- }}^{K_{\eta_a}}
V_{\eta_u} \,\ud u.
\]
Therefore, $\int_{{\Real_+}} V_t \,\ud K_t = \int_{[0, 1)} V_{\eta_u}
\indic
_{\{\eta_u < \infty\}} \,\ud u$ follows. The last fact helps to
establish the
following result.

\begin{prop} \label{prop: expect wrt Q}
Let $\rho$ be a random time on $(\Omega, \F, \bF, \prob)$. Then, 
for any nonnegative $(\Omega, \bF)$-optional process $V$, it holds that
%
\begin{equation}
\label{eq: change of variables} \expec_\prob[V_\rho] = \int
_{[0, 1)} \expec_{\mathbb
{Q}_u}[V_{\eta_u} ] \,\ud u = \lim
_{a \uparrow1} \expec_{\mathbb{Q}_a} \biggl[\int_{[0, a]}
V_{\eta_u} \,\ud u \biggr].
\end{equation}
\end{prop}

\begin{pf}
As discussed above, for any $V$ that is nonnegative and $(\Omega, \bF)
$-optional, we have
\[
\int_{{\Real_+}} V_t L_t \,\ud
K_t = \int_{[0, 1)} V_{\eta_u}
L_{\eta_u} \indic _{\{\eta_u < \infty\}} \,\ud u.
\]
Therefore, the first equality in \eqref{eq: change of variables} is
immediate from Theorem~\ref{thmm: doleans sharpened}, Fubini's theorem,
the definition of the probabilities $(\qprob_u)_{u \in[0, 1)}$ and Lemma~\ref{lem: K under Q}. The second equality in \eqref{eq: change of
variables} follows from the monotone convergence theorem and the
consistency of the family $(\qprob_u)_{u \in[0, 1)}$.
\end{pf}

Proposition~\ref{prop: expect wrt Q} has a simple corollary, which
states that the law of $K_{\rho-}$ under $\prob$ is stochastically
dominated (in first order) by the standard uniform law, and that the
latter standard uniform law is stochastically dominated by the law of
$K_{\rho}$ under~$\prob$.

\begin{prop} \label{prop: unif domination}
Let $\rho$ be any random time on $(\Omega, \F, \bF, \prob)$ with
associated pair $(K,
L)$. Then, for all nondecreasing functions $f\dvtx[0, 1)\mapsto\Real
$, it
holds that
%
\begin{equation}
\label{eq: stoch dom for K_rho} \expec_\prob \bigl[f(K_{\rho-}) \bigr] \leq\int
_{[0, 1)} f(u) \,\ud u \leq \expec_\prob
\bigl[f(K_{\rho}) \bigr].
\end{equation}
\end{prop}

\begin{pf}
Pick any nondecreasing function $f \dvtx[0, 1)\mapsto\Real$. For
establishing the inequalities \eqref{eq: stoch dom for K_rho}, it is
clearly sufficient to deal with the case where $f(u) \in\Real_+$ for
$u \in[0, 1)$. Since $K_{\eta_u - } \leq u$ and $f$ is nondecreasing,
\eqref{eq: change of variables} gives
\[
\expec_\prob \bigl[f(K_{\rho-}) \bigr] = \int
_{[0, 1)} \expec_{\mathbb
{Q}_u} \bigl[f (K_{\eta_u - })
\bigr] \,\ud u \leq\int_{[0, 1)} \expec_{\mathbb{Q}_u} \bigl[f (u)
\bigr] \,\ud u = \int_{[0, 1)} f (u) \,\ud u.
\]
The other inequality in \eqref{eq: stoch dom for K_rho} is proved
similarly, using the fact that ${\mathbb{Q}_u}[K_{\eta_u} \geq u] = 1$
for $u \in[0, 1)$, as follows from Lemma~\ref{lem: K under Q}.
\end{pf}

\subsection{Random times that avoid all stopping times} \label{subsec:
avoid_stop}

A random time $\rho$ on $(\Omega, \F, \bF, \prob)$ is said to
\textsl{avoid all
stopping times on $(\Omega, \bF, \prob)$} if $\prob[\rho= \tau] =
0$ holds
whenever $\tau$ is a stopping time on $(\Omega, \bF)$. The next
result states
equivalent conditions to $\rho$ avoiding all stopping times.

\begin{prop} \label{prop: equiv for rand time avoiding all stop times}
Let $\rho$ be any random time on $(\Omega, \F, \bF, \prob)$ with
associated canonical
pair $(K, L)$. Then the following statements are equivalent:
\begin{longlist}[(1)]
\item[(1)]$\rho$ avoids all stopping times on $(\Omega, \bF, \prob)$.
\item[(2)]$\Delta K$ is $\prob$-evanescent.
\item[(3)]$\prob[\Delta K_\rho= 0] = 1$.
\item[(4)]$K_\rho$ has the standard uniform distribution under $\prob$.
\end{longlist}
\end{prop}

\begin{pf}
In the course of the proof, we shall be using $A$, $Z$ and $N$ for the
processes that were introduced in Section~\ref{subsec: canon pair},
associated to the random time $\rho$ on $(\Omega, \F, \bF, \prob)$.

For implication $(1) \Rightarrow(2)$, the fact that $\expec_\prob
[\Delta
A_\tau] = \prob[\rho= \tau] = 0$ implies that $\prob[\Delta A_\tau=
0] = 1$ holds for all stopping times $\tau$ on $(\Omega, \bF)$.
Then, in view
of~\eqref{eq: defn of K}, $\prob[\Delta K_\tau= 0] = 1$ holds for all
stopping times $\tau$ on $(\Omega, \bF)$ as well. An application
of~\cite{MR1219534}, Theorem~4.10, shows that $\Delta K$ is $\prob$-evanescent.
Implication $(2) \Rightarrow(3)$ is trivial. Now, assume (3); from the
inequalities \eqref{eq: stoch dom for K_rho} we get $\expec[f(K_\rho)]
= \int_{[0, 1)} f(u) \,\ud u$ for any nondecreasing Borel function $f
\dvtx[0, 1)
\mapsto\Real_+$, which implies that $K_\rho$ has a standard uniform
distribution under $\prob$. In the next three paragraphs, we shall show
$(4) \Rightarrow(3) \Rightarrow(2) \Rightarrow(1)$.

We show $(4) \Rightarrow(3)$. By \eqref{eq: change of variables}, we have
\begin{eqnarray*}
\expec_\prob[K_{\rho} + K_{\rho-}] = \lim
_{a \uparrow1} \expec _{\mathbb{Q}_a} \biggl[\int_{[0, a]}
(K_{\eta_u} + K_{\eta_u -}) \,\ud u \biggr].
\end{eqnarray*}
For $a \in[0, 1)$, on the event $\{K_{\eta_a} \geq a\}$ it holds that
\[
a^2 = \int_{[0, a]} 2 u \,\ud u \leq\int
_{[0, a]} (K_{\eta_u} + K_{\eta_u -}) \,\ud u \leq1.
\]
With the help of Lemma~\ref{lem: K under Q}, we obtain $\expec_\prob
[K_{\rho} + K_{\rho-}] = 1$. Since $\expec_\prob[K_{\rho}] = 1/2$ holds
in view of the fact that $K_\rho$ has the standard uniform distribution
under $\prob$, we obtain $\expec[K_{\rho-}] = 1/2$. As $K$ is
nondecreasing and $\expec_\prob[\Delta K_\rho] = 0$, we obtain
$\prob
[\Delta K_{\rho} = 0] = 1$, that is, statement $(3)$.

For $(3) \Rightarrow(2)$, start with the following observation: for
any stopping time $\tau$, on $\{\tau< \infty\}$ it holds that
\begin{eqnarray*}
L_\tau&= &L_{\tau-} + \Delta L_\tau= L_{\tau-}
+ \frac{\Delta N_\tau}{1
- K_{\tau-}} \\
&=& \frac{L_{\tau-} (1 - K_{\tau-}) +Z_{\tau} - Z_{\tau
-} +
\Delta A_\tau}{1 - K_{\tau-}} = \frac{Z_{\tau} + \Delta A_\tau}{1 -
K_{\tau-}}.
\end{eqnarray*}
Since $\{\Delta K_\tau> 0\} \subseteq\{\Delta A_\tau> 0\}$ holds
on $\{\tau< \infty\}$, it follows that $\{\Delta K_\tau> 0\}
\subseteq\{L_\tau> 0\}$ modulo $\prob$ holds on $\{\tau< \infty\}
$ for all stopping times $\tau$. Continuing, note that
\begin{eqnarray*}
0 &=& \expec_\prob[\Delta K_{\rho}] = \expec_\prob
\biggl[\int_{{\Real
_+}} (K_{t} - K_{t-})
L_t \,\ud K_t \biggr] \\
&=& \expec_\prob \biggl[\sum
_{t \in{\Real
_+}} L_t (\Delta K_t)^2
\biggr].
\end{eqnarray*}
Consider a sequence $(\tau_n)_{n \in\Natural}$ of stopping times
with disjoint
graphs that exhausts the jumps of $K$; then, $\expec_\prob[\sum_{n
\in\Natural} L_{\tau_n} (\Delta K_{\tau_n})^2] = 0$. This means
that $\sum_{n \in\Natural}
L_{\tau_n} (\Delta K_{\tau_n})^2 = 0$, $\prob$-a.s.; since $\{
\Delta K_{\tau_n} > 0\} \subseteq\{L_{\tau_n} > 0\}$ modulo $\prob$
holds on
$\{\tau_n < \infty\}$ for all $n \in\Natural$, we obtain $\prob
[\Delta K_{\tau_n} = 0] = 1$ for all $n \in\Natural$. The last
implies that $\prob[
\Delta K_\tau= 0] = 1$ for all stopping times $\tau$. In view of
\cite{MR1219534}, Theorem~4.10, this is exactly statement (2).

Finally, we establish $(2) \Rightarrow(1)$. Since
\begin{eqnarray*}
\{\Delta A_\tau> 0\} &=& \{L_\tau\Delta K_\tau> 0\}
= \{L_\tau> 0\} \cap\{\Delta K_\tau> 0\} \\
&=& \{\Delta
K_\tau> 0\}
\end{eqnarray*}
modulo $\prob$ holds for all stopping times $\tau$, we have $\prob
[\rho
= \tau] = \expec_\prob[\Delta A_\tau] = 0$, the latter being valid because
$\prob[\Delta A_\tau> 0] = \prob[\Delta K_\tau> 0] = 0$. Therefore,
$\rho$ avoids all stopping times under $\prob$.
\end{pf}

\subsection{An optimality property of $L$ amongst all nonnegative
local $\mathbb{P}$-mar\-tin\-gales} \label{subsec: num property}

Let $\Sl$ be the set of all nonnegative super-martingales $S$ on
$(\Omega, \bF, \prob)$ with $\prob[S_0 = 1] = 1$. The set $\Sl$
contains in
particular all nonnegative local martingales $M$ on $(\Omega, \bF,
\prob)$ with
$\prob[M_0 = 1] = 1$. For a random time $\rho$ with associated
canonical pair $(K, L)$, it is reasonable to expect that the local
martingale $L$ has some optimality property within the class $\Sl$ when
sampled at $\rho$. Indeed, the next result shows that, in the jargon of
\cite{Kar09}, $L_\rho$ is the num\'eraire under $\prob$ in the
convex set
$\{S_\rho| S \in\Sl\}$.

\begin{prop} \label{prop: num prof of L}
Let $\rho$ be a random time on $(\Omega, \F, \bF, \prob)$ with
associated canonical
pair $(K, L)$. Then $\prob[L_\rho> 0] = 1$ and $\expec_\prob[S_\rho
/ L_\rho] \leq1$ holds for all $S \in\Sl$. If, furthermore, $\rho$
avoids all stopping times on $(\Omega, \bF, \prob)$, then the
stronger inequality
$\expec_\prob[S_\rho/ L_\rho| K_\rho] \leq1$ holds for all $S
\in\Sl$.
\end{prop}

\begin{pf}
By Lemma~\ref{lem: K under Q}, ${\mathbb{Q}_u}[L_{\eta_u} > 0] = 1$ holds
for all $u \in[0, 1)$. Then, by Proposition~\ref{prop: expect wrt Q},
\[
\prob[L_\rho> 0] = \int_{[0, 1)} {
\mathbb{Q}_u}[L_{\eta_u} > 0] \,\ud u = 1.
\]

Fix $S \in\Sl$. Observe that $\expec_{\mathbb{Q}_u}[S_{\eta_u} /
L_{\eta_u}] =
\expec_\prob[S_{\eta_u} \indic_{\{L_{\eta_u} > 0\}}] \leq1$ holds for
all $u \in[0, 1)$. Then
\[
\expec_\prob[S_\rho/ L_\rho] = \int
_{[0, 1)} \expec_{\mathbb
{Q}_u}[S_{\eta_u} /
L_{\eta_u}] \,\ud u \leq1.
\]
Assume now that $\rho$ avoids all stopping times on $(\Omega, \bF,
\prob)$. By a
straightforward extension of Lemma~\ref{lem: K under Q}, ${\mathbb
{Q}_u}[K_{\eta_u} = u] = 1$ holds for all $u \in[0, 1)$. Therefore,
for all
functions $f \dvtx[0, 1)\mapsto\Real_+$,
\begin{eqnarray*}
\expec_\prob \bigl[(S_\rho/ L_\rho)
f(K_\rho) \bigr] &=& \int_{[0, 1)}
\expec_{\mathbb{Q}_u} \bigl[(S_{\eta_u} / L_{\eta_u})
f(K_{\eta_u}) \bigr] \,\ud u
\\
&=& \int_{[0, 1)} \expec_{\mathbb{Q}_u} \bigl[(S_{\eta_u}
/ L_{\eta_u}) f(u) \bigr] \,\ud u
\\
&\leq&\int_{[0, 1)} f(u) \,\ud u = \expec_\prob
\bigl[f(K_\rho) \bigr],
\end{eqnarray*}
the last equality following from Proposition~\ref{prop: equiv for rand
time avoiding all stop times}. Since the function $f \dvtx[0,
1)\mapsto
\Real
_+$ is arbitrary, we obtain $\expec_\prob[S_\rho/ L_\rho| K_\rho]
\leq1$.
\end{pf}

\section{Random times and randomised stopping times} \label{sec: equal
in law}

\subsection{The one probability $\qprob$}

Recall the consistent family of probabilities $(\qprob_u)_{u \in[0, 1)}$
from Section~\ref{subsec: consist fam prob assoc with rand}. For the
purposes of Section~\ref{sec: equal in law}, we shall be working under
the following assumption.

\begin{ass} \label{ass: existence of prob}
There exists a probability measure $\qprob\equiv\qprob_1$ on
$(\Omega, \F)$, as well as a random variable $U \dvtx\Omega\mapsto[0, 1)$,
such that:
\begin{longlist}[(1)]
\item[(1)]$\qprob|_{\F_{\eta_u}} = \qprob_u|_{\F_{\eta_u}}$ holds for all
$u \in[0, 1)$.
\item[(2)] Under $\qprob$, $U$ is independent of $\F_{\infty}$ and has the
standard uniform law.
\end{longlist}
\end{ass}

\begin{rem} \label{rem: no loss of gen the unif}
Given that there exists a probability measure $\qprob\equiv\qprob_1$
on $(\Omega, \F)$ such that $\qprob|_{\F_{\eta_u}} = \qprob
_u|_{\F
_{\eta_u}}$ holds for all $u \in[0, 1)$, asking that there also
exists a
random variable $U \dvtx\Omega\mapsto[0, 1)$ such that $U$ is independent
of $\F_{\infty}$ and has the standard uniform law under $\qprob$
entails no loss of generality whatsoever. Indeed, if such random
variable does not exist, the underlying probability space can always be
enlarged in order to support one. More precisely, define $\oOmega:=
\Omega\times[0, 1)$, a filtration $\overline{\bF}= (\oF_t)_{t \in
\Real_+}$ via
$\oF_t = \F_t \otimes\{\varnothing, [0, 1)\}$ for $t \in\Real_+$, as
well as $\oF= \F\otimes\B([0, 1))$, where $\B([0, 1))$ is the Borel
sigma-algebra on $[0, 1)$. It is immediate that $(\F_t)_{t \in\Real_+}$
and $(\oF_t)_{t \in\Real_+}$ are in one-to-one correspondence.
(However, $\F$ and $\oF$ are not isomorphic.) On $(\oOmega, \oF)$,
define $\oprob:= \prob\otimes\Leb$, $\overline{\qprob}:= \qprob
\otimes
\Leb$, as well as $\overline{\qprob}_u:= \qprob_u \otimes\Leb$
for $u \in
[0, 1)
$, where ``$\Leb$'' denotes Lebesgue measure on $\B([0, 1))$. Any process
$X$ on the original stochastic basis is identified on the new
stochastic basis with the process $\oX$ defined via $\oX(\omega, u) =
X(\omega)$ for all $(\omega, u) \in\oOmega$---this way, properties
like adaptedness and optionality of processes are in one-to-one
correspondence. The random variable $U \dvtx\oOmega\mapsto[0, 1)$ defined
via $U(\omega, u) = u$ for all $(\omega, u) \in\oOmega$ has the
standard uniform distribution, and is independent of $\oF_{\infty}$,
the previous holding under both $\oprob$ and $\overline{\qprob}$.
Note that the
pair associated with $\rho$ on $(\oOmega, \oF, \overline{\bF},
\oprob)$ is $(\oK, \oL)$ in the
previously-introduced notation, which is identified with $(K, L)$.
Furthermore, $\overline{\qprob}|_{\oF_{\eta_u}} = \overline{\qprob
}_u|_{\oF_{\eta_u}}$
holds for all $u \in[0, 1)$.
\end{rem}

\begin{rem} \label{rem: prob_not_always_there}
Even though item (2) of Assumption~\ref{ass: existence of prob} is not
really an assumption in view of Remark~\ref{rem: no loss of gen the
unif} above, item (1) \emph{is}, as Example~\ref{exa:
need_for_no_compl} will reveal. In fact, Example~\ref{exa:
need_for_no_compl} will make an additional point: even if $\qprob$
exists, it is in general possible that neither of the conditions
$\qprob\ll_{\F_t} \prob$ nor $\prob\ll_{\F_t} \qprob$ holds,
for any
choice of $t \in(0, \infty)$. This clarifies the absolute need to
refrain from completing $\bF= (\F_t)_{t \in\Real_+}$ with $\prob
$-null sets,
even if the null sets come from $\bigcup_{t \in\Real_+} \F_t$ and not
from the larger, in general, sigma-field $\F_\infty= \bigvee_{t \in
\Real_+} \F_t$.
\end{rem}

\subsection{The stochastic behaviour of optional processes up to
random times}

We now turn to the topic discussed in the \hyperref[sec1]{Introduction}: as long
as distributional properties of optional processes on $(\Omega, \bF)$
up to a
random time are concerned, one can pass from the original random time
$\rho$ and probability $\prob$ to a randomised stopping time $\psi$ on
$(\Omega, \bF, \qprob)$, where $\qprob$ is the probability of
Assumption~\ref{ass:
existence of prob}.

\begin{thmm} \label{thmm: law equal rand pseudo}
Let $\rho$ be a random time on $(\Omega, \F, \bF, \prob)$ with
associated canonical
pair $(K, L)$. Under the validity of Assumption~\ref{ass: existence of
prob}, let $\qprob$ the probability that appears there. Define
\[
\psi:= \inf\{t \in{\Real_+}| K_t \geq U\} = \eta_U.
\]
Then $\psi$ is a randomised stopping time on $(\Omega, \F, \bF,
\qprob)$ with
associated canonical pair $(K, 1)$. Furthermore, for any optional
process $Y$ on $(\Omega, \bF)$, the finite-dimensional distributions of
$Y^\rho= (Y_{\rho\wedge t})_{t \in\Real_+}$ under $\prob$ coincide
with the
finite-dimensional distributions of $Y^\psi= (Y_{\psi\wedge t})_{t
\in\Real_+
}$ under $\qprob$.
\end{thmm}

\begin{pf}
Observe that $\{\psi> t\} = \{U > K_t\}$ holds for $t \in{\Real_+}$.
Therefore,
\[
\qprob[\psi> t | \F_t] = \qprob[U > K_t |
\F_t] = 1 - K_t\qquad \mbox{for } t \in{\Real_+}.
\]
It follows that the pair associated with $\psi$ on $(\Omega, \bF,
\qprob)$ is $(K, 1)$.

Pick any nonnegative optional process $V$ on $(\Omega, \bF)$. Then
%
\begin{eqnarray}
\label{eq: P and Q fdd equiv pre} \expec_\prob[V_\rho] &=& \int
_{[0, 1)} \expec_{\mathbb
{Q}_u}[V_{\eta_u} ] \,\ud u = \int
_{[0, 1)} \expec_\qprob[V_{\eta_u} ] \,\ud u
\nonumber
\\[-8pt]
\\[-8pt]
\nonumber
&=&
\expec_\qprob \biggl[\int_{[0, 1)} V_{\eta_u} \,\ud
u \biggr] = \expec_\qprob[V_{\eta_U}] = \expec
_\qprob[V_\psi].
\end{eqnarray}
Continuing, fix an optional process $Y$ on $(\Omega, \bF)$ and times
$\{t_1,\ldots, t_n\} \subseteq\Real_+$. For any nonnegative
Borel-measurable function $f\dvtx\Real^n \mapsto\Real_+$, the
process $V
= f(Y^{t_1},\ldots, Y^{t_n})$ is optional on $(\Omega, \bF)$. Since
$V_\rho=
f(Y^{\rho}_{t_1},\ldots, Y^\rho_{t_n})$ and $V_\psi= f(Y^{\psi
}_{t_1},\ldots, Y^{\psi}_{t_n})$, \eqref{eq: P and Q fdd equiv pre} gives
\[
\expec_\prob \bigl[f \bigl(Y^{\rho}_{t_1},\ldots,
Y^\rho_{t_n} \bigr) \bigr] = \expec _\qprob \bigl[f
\bigl(Y^\psi_{t_1},\ldots, Y^\psi_{t_n}
\bigr) \bigr].
\]
As the collection $\{t_1,\ldots, t_n\} \subseteq\Real_+$ and the
nonnegative Borel-measurable function $f$ are arbitrary, the
finite-dimensional distributions of $Y^\rho$ under $\prob$ coincide
with the finite-dimensional distributions of $Y^\psi$ under $\qprob$.
\end{pf}

\begin{rem} \label{rem: tail_prob}
In the setting of Theorem~\ref{thmm: law equal rand pseudo}, assume
that $\tau$ is a stopping time on $(\Omega, \bF)$ and that $E$ is an
$\F
_\tau
$-measurable set. Then, since the process $\indic_E \indic_{]\!] \tau, \infty[\![}$ is optional, we obtain
\begin{eqnarray*}
\prob[E, \rho> \tau] &=& \qprob[E, \eta_U > \tau] = \qprob[E,
K_\tau< U] \\
&=&\int_{[0, 1)} \qprob[E, K_\tau<
u] \,\ud u = \expec_\qprob \bigl[(1 - K_\tau)
\indic_E \bigr].
\end{eqnarray*}
\end{rem}

\section{First examples} \label{sec: examples}

\subsection{Finite-horizon discrete-time models} \label{subsec: discr time}
Models where the time-set is discrete can be naturally embedded in a
continuous-time framework. \emph{Only for the purposes of Section~\ref
{subsec: discr time}}, we consider a filtered probability space
$(\Omega, \F, \bF, \prob)$ with $\bF= (\F_t)_{t \in\Time}$,
where $\Time= \{0,\ldots, T\}$ for $T \in\Natural$. We assume that
$\F= \F_T \vee
\sigma
(U)$, where $U$ is a random variable with uniform distribution under
$\prob$, independent of $\F_T$. A random time $\rho$ in this setting is
a $\Time$-valued random variable.

It is straightforward to check that $A = \sum_{t \leq\cdot} \prob
[\rho
= t | \F_t]$ is the dual optional projection on $(\Omega, \bF, \prob
)$ of
$\indic
_{[\![ \rho, T ]\!]}$. Recall from Section~\ref{subsec: canon pair} the
stopping time $\zeta_0:= \min\{t \in\Time| Z_t = 0\}$. The
discrete-time versions of \eqref{eq: defn of K} and \eqref{eq: defn of
L} on $\{t \leq\zeta_0\}$ read
\[
K_t = K_{t-1} + (1 - K_{t-1}) \biggl(
\frac{A_t - A_{t-1}}{Z_t + A_t -
A_{t-1}} \biggr) = K_{t-1} + (1 - K_{t-1})
\frac{\prob[\rho= t | \F
_t]}{\prob[\rho\geq t | \F_t]}
\]
and
\[
L_t = L_{t-1} \biggl(1 + \frac{N_t - N_{t-1}}{Z_{t-1}} \biggr) =
L_{t-1} \frac
{Z_t + A_t - A_{t-1}}{Z_{t-1}} = L_{t-1} \frac{\prob[\rho\geq t
|
\F_t]}{\prob[\rho\geq t | \F_{t-1}]}.
\]
On $\{t > \zeta_0\}$, $K_t = K_{\zeta_0}$ and $L_t = L_{\zeta_0}$ holds.

In finite-horizon discrete-time settings like the one considered here,
nonnegative local martingales are actually martingales; see \cite
{MR1809522}. Therefore, one may define a probability $\qprob$ on
$(\Omega, \F)$ that has density $L_T$ with respect to $\prob$; then,
$\qprob|_{\F_{\eta_u}} = {\mathbb{Q}_u}|_{\F_{\eta_u}}$ holds for
all $u
\in
[0, 1)$. The probability $\qprob$ is absolutely continuous with
respect to
$\prob$. (Observe also that Assumption~\ref{ass: existence of prob} is
always valid in this setting. Indeed, $L_T$ is $\F_T$-measurable and,
therefore, independent of $U$ under $\prob$, which implies that $U$ is
independent of $\F_T$ under $\qprob$.) The next result shows that the
stochastic behaviour of $\rho$ under $\prob$ and $\qprob$ might be
radically different.

\begin{prop} \label{prop: discr time rand}
Let $\rho$ be a random time on $(\Omega, \bF, \prob)$. If $\prob
[\rho= \zeta_0
| \F_{\zeta_0}]$ is $\prob$-a.s. $\{0, 1\}$-valued, then
$\qprob
[\rho= \zeta_0] = 1$.
\end{prop}

\begin{pf}
On $\{\zeta_0 > 0\}$ it holds that $L_{\zeta_0} = L_{\zeta_0 - 1}
\prob[\rho= \zeta_0 | \F_{\zeta_0}] / \prob[\rho= \zeta_0
|\break   \F
_{\zeta_0-1}]$, which implies that $\{L_{\zeta_0} > 0\} = \{\prob
[\rho= \zeta_0 | \F_{\zeta_0}] > 0\}$. Since $\prob[\rho=
\zeta_0
| \F_{\zeta_0}]$ is $\prob$-a.s. $\{0, 1\}$-valued, it follows
that $\{L_{\zeta_0} > 0\} = \{\prob[\rho= \zeta_0 | \F_{\zeta
_0}] = 1\}$ holds modulo $\prob$ on $\{\zeta_0 > 0\}$. On $\{\zeta
_0 = 0\}$ both $L_{\zeta_0} = 1$ and $\prob[\rho= \zeta_0 | \F
_{\zeta_0}] = 1$ hold modulo~$\prob$. Therefore,
\[
\qprob[\rho= \zeta_0] = \expec_\prob[L_{\zeta_0}
\indic_{\{\rho
= \zeta_0\}}] = \expec_\prob \bigl[L_{\zeta_0} \prob[\rho=
\zeta_0 | \F_{\zeta_0}] \bigr] = \expec_\prob[L_{\zeta_0}]
= 1,
\]
which completes the proof.
\end{pf}

Random times that satisfy the condition in the statement of Proposition~\ref{prop: discr time rand} are $\qprob$-a.s. equal to a stopping time.
The next example shows that familiar random times that are far from
being stopping times under $\prob$ become $\qprob$-a.s. equal to a constant.

\begin{exa} \label{exa: sup in discrete time}
Let $X$ be an adapted process on $(\Omega, \F, \bF, \prob)$ such
that $\prob[X_t
\geq
X_{t-1} | \F_{t-1}] > 0$ holds $\prob$-a.s. for all $t \in\Time
\setminus\{0\}$. Define $\rho:= \max \{ t \in\Time| X_t
= X^\up_T  \}$ to be the last time of maximum of $X$. On the event
$\{\zeta_0 < T\}$, and in view of $\prob[X_{\zeta_0 + 1} \geq
X_{\zeta
_0} | \F_{\zeta_0}] > 0$ holding $\prob$-a.s., we have $\prob
[\rho
= \zeta_0 | \F_{\zeta_0}] = 0$ holding $\prob$-a.s. On the other
hand, on the event $\{\zeta_0 = T\}$ we have $\prob[\rho= \zeta_0
| \F_{\zeta_0}] = \indic_{\{\rho= T\}}$, which is $\prob$-a.s.
$\{0,1\}$-valued. From Proposition~\ref{prop: discr
time rand}, it follows that $\qprob[\rho= \zeta_0] = 1$. Since
$\prob
[\rho= \zeta_0 < T] = 0$ and $\qprob$ is absolutely continuous with
respect to $\prob$, we obtain $\qprob[\rho= T] = 1$.
\end{exa}

A continuous-time version of Example~\ref{exa: sup in discrete time}
involving Brownian motion with drift over finite time-intervals will be
given in Section~\ref{subsec: time_max_BM_fin}, where it will be
demonstrated in particular that the corresponding probabilities $\prob$
and $\qprob$ in that setting are singular.

\subsection{Time of maximum of nonnegative local martingales with zero
terminal value, continuous running supremum and no jumps while at their
running supremum} \label{subsec: sup of cont local marts}

For special cases of random times, the calculation of the canonical
pair becomes relatively easy. More information and extensive discussion
on the material of Section~\ref{subsec: sup of cont local marts} can be
found in \cite{Kar_honest}, where exact connections with so-called
\emph
{honest times} are presented.

Let us introduce some notation: $\mathcal{L}_0$ denotes the class of all
nonnegative local martingales $M$ such that $\prob[M_0 = 1, M_\infty
= 0] = 1$ (where $M_\infty:= \lim_{t \to\infty} M_t$,
noting that the limit in the definition of $L_\infty$ exists in the
$\prob$-a.s. sense, in view of the nonnegative super-martingale
convergence theorem), the running supremum process $M^* = M^\uparrow$
is continuous and $\{M_- = M^*_-\} \subseteq\{\Delta M = 0\}$
holds up to a $\prob$-evanescent set. For $M \in\mathcal{L}_0$, define
%
\begin{equation}
\label{eq: rl} \rho_M:= \sup \bigl\{t \in\Real_+| M_{t-} =
M^*_{t-} \bigr\}.
\end{equation}
(The convention $M_{0-} = 0 = M^*_{0-}$ implies that the random set
$\{t \in\Real_+| M_{t-} = M^*_{t-}\}$ is nonempty.) Since $\prob
[M_\infty= 0] = 1$ holds for $M \in\mathcal{L}_0$, it follows that
$\prob[\rho_M< \infty] = 1$. Whenever $M \in\mathcal{L}_0$, it
$\prob$-a.s. holds that
$M_{\rho_M-} = M_{\rho_M} = M^*_{\rho_M}$; in fact, as
\cite{Kar_honest}, Theorem~1.2, implies, the previous random variables are
also equal
to $M^*_{\infty}$, which makes $\rho_M$ a time of \emph{overall} maximum
of \mbox{$M \in\mathcal{L}_0$}.

\begin{prop} \label{prop: pair for max}
Let $M \in\mathcal{L}_0$, and let $\rho$ be any time of maximum of
$M$, in the
sense that $\prob[M_\rho= M^*_\infty] = 1$. Then the following are true:
\begin{itemize}
\item The canonical pair associated with $\rho$ is $(K, L) = (1 - 1 /
M^*, M)$.
\item$\rho$ avoids all stopping times on $(\Omega, \bF, \prob)$.
\item$\prob[\rho= \rho_M] = 1$.
\end{itemize}
\end{prop}

\begin{pf}
Only a sketch of the proof is provided; as already mentioned, more
information can be found in \cite{Kar_honest}. Note that $\prob[\rho
\leq\rho_M] = 1$ holds by definition on $\rho_M$; in particular,
$\prob[\rho< \infty] = 1$. The fact that $\rho$ avoids all
stopping times on $(\Omega, \bF, \prob)$ follows from Doob's maximal
identity, as
presented in \cite{MR2247846}; more precisely, $\prob[\rho= \tau|
\F_{\tau}] = 0$ holds on $\{\tau< \infty, M_{\tau} < M^*_{\tau}\}
$, while on $\{\tau< \infty, M_{\tau} = M^*_{\tau}\}$ it
follows that
\[
\prob[\rho= \tau| \F_{\tau}] = \prob \Bigl[\sup_{t \in[\tau, \infty
)}
M_t > M_\tau \big| \F_\tau \Bigr] = 1 -
\frac{M_\tau
}{M^*_\tau} = 0.
\]
Doob's maximal identity applied again implies that $Z = M / M^*$ (see
\cite{MR2247846}); then, since $\rho$ avoids all stopping times on
$(\Omega, \bF, \prob)$, one can use Remark~\ref{rem: L as mult
decomp} to conclude
that the canonical pair associated with $\rho$ is $(1 - 1/M^*, M)$.

Since $\rho_M$ is a special instance of a random time that achieves the
maximum of~$M$, it follows that the pair associated with $\rho_M$ is
also $(1 - 1/M^*, M)$. Since the canonical pair associated to a random
time completely determines its distribution, the laws of $\rho$ and
$\rho_M$ are the same under $\prob$. Combined with $\prob[\rho\leq
\rho_M
] = 1$, we obtain $\prob[\rho= \rho_M] = 1$.
\end{pf}

\begin{rem} \label{rem: unique max}
Proposition~\ref{prop: pair for max} implies in particular that there
exists an almost surely unique time of maximum of processes in
$\mathcal{L}_0$.
\end{rem}

\begin{rem} \label{rem: K less than one, L strict loc mart }
It was already hinted out in the discussion at Section~\ref{subsec:
canon pair} that the canonical pair $(K, L)$ associated with a random
time may be such that $\prob[K_\infty< 1] > 0$ holds;
additionally, $L$ may fail to be a true martingale. Indeed, in the
context of Proposition~\ref{prop: pair for max}, $M = L$ can be freely
chosen to be a strict local martingale in the terminology of \cite
{MR1478722}; furthermore, $\prob[K_\infty< 1] = \prob[L^*_\infty<
\infty] = 1$.
\end{rem}

\begin{rem}
Recall the set $\Sl$ from Section~\ref{subsec: num property}.
Specialising to the setting of Proposition~\ref{prop: pair for max},
let $\rho$ be the time of maximum of $M \in\mathcal{L}_0$. In this
case, and
since $K_\rho= 1 - 1 / M_\rho$, we obtain from Proposition~\ref{prop:
num prof of L} that $\expec_\prob[S_\rho| M_\rho] \leq M_\rho$ for all
$S \in\Sl$. This result is quite interesting---it states that \emph{no
matter} what the level of $M$ at its maximum, no other nonnegative
super-martingale with unit initial value is expected to lie above that.

Since $\Sl$ is convex, the condition $\expec_\prob[S_\rho| M_\rho]
\leq M_\rho$ for all $S \in\Sl$ is actually equivalent to the fact
that $M_\rho$ stochastically dominates all random variables in $\{
S_\rho| S \in\Sl\}$ in second order, meaning that $\expec_\prob
[U(S_\rho)] \leq\expec_\prob[U(M_\rho)]$ holds for all nondecreasing
concave functions $U\dvtx\Real_+ \mapsto\Real$. In fact, a stronger
statement is true. Since $S$ is a nonnegative super-martingale on
$(\Omega, \bF, \prob)$ with $\prob[S_0 = 1] = 1$ for all $S \in\Sl
$, Doob's maximal
inequality implies that $\prob[S_\rho> x] \leq1 \wedge(1 / x)$
holds for all $x \in(0, \infty)$. On the other hand, since $M \in
\mathcal{L}_0
$, it follows from Doob's maximal identity \cite{MR2247846} that
$\prob
[M_\rho> x] = 1 \wedge(1 / x)$ holds for all $x \in(0, \infty)$.
Therefore, $\sup_{S \in\Sl} \prob[S_\rho> x] = \prob[M_\rho>
x]$ holds for all $x \in(0, \infty)$, which implies that $M_\rho$
stochastically dominates all random variables in $\{S_\rho| S \in\Sl
\}$, even in first order.
\end{rem}

\begin{exa} \label{exa: BM - sup}
Let $\Omega$ be the canonical space of continuous functions from
$\Real
_+$ to $\Real$. Take $X$ to be the coordinate process and $\bF$ be the
right-continuous augmentation of the natural filtration of $X$. For the
time being, $\F$ is taken to be equal to $\F_\infty$. Let $\prob$ be
the unique probability on $(\Omega, \F)$ under which $X$ is a Brownian
motion with (strictly negative) drift $\mu< 0$ and unit diffusion
coefficient. Since $\prob[\lim_{t \to\infty} X_t = - \infty] =
1$, consider a random time $\rho$ that is a time of overall maximum of
$X$. Note that $\rho$ is also a time of maximum of the process $M:=
\exp(- 2 \mu X)$, which satisfies all the conditions of Proposition~\ref
{prop: pair for max}. We obtain that the canonical representation pair
$(K, L)$ of $\rho$ on $(\Omega, \F, \bF, \prob)$ is such that $K =
1 - \exp(2 \mu
X^\up
)$ and $L = \exp(- 2 \mu X)$. An application of Proposition~\ref{prop:
equiv for rand time avoiding all stop times} gives that $\sup_{t \in
\Real_+} X_t = (1 / 2 \mu) \log(1 - K_\rho)$ has the exponential
distribution with rate $- 2 \mu$ under $\prob$---of course, this fact
is well known.

Note that the process $L = \exp(- 2 \mu X)$ is a martingale on
$(\Omega, \bF, \prob)
$. Since we are working on the canonical space, a joint application of
the extension theorem of Daniell--Kolmogorov \cite{MR1121940},
Section~2.2A, and Girsanov's theorem
\cite{MR1121940}, Section~3.5,
imply there exists a probability $\qprob$ on $(\Omega, \F, \bF)$
such that
$\,\ud
\qprob= L_t \,\ud\prob$ holds on each $\F_t$ for $t \in\Real_+$,
and under
which $X$ is a Brownian motion with drift $- \mu> 0$ and unit
diffusion coefficient. In order to be in par with Assumption~\ref{ass:
existence of prob}, we carry out the enlargement of the probability
space as discussed in Remark~\ref{rem: no loss of gen the unif}. Then
it comes as a consequence of Theorem~\ref{thmm: law equal rand pseudo}
that a path of $X^\rho$ under $\prob$ can be stochastically realised
as follows:
\begin{longlist}[(1)]
\item[(1)] With $U$ being a standard uniform random variable, set $X^\up
_\infty= X_\rho= (1 / 2 \mu) \log(U)$.
\item[(2)] Given $x = X_\rho$, generate $X^{\tau_x}$ under $\qprob$, where
$\tau_x:= \inf\{t \in\Real_+ | X_t = x\}$.
\end{longlist}
\end{exa}

The next example will settle a couple of claims that were previously
made in Remark~\ref{rem: prob_not_always_there}.

\begin{exa} \label{exa: need_for_no_compl}
Consider the interval $(0, \infty)$, with an extra ``cemetery'' state
$\triangle$ appended in a way so that $\triangle$ is a topologically
isolated point of $(0, \infty) \cup\{\triangle\}$. For a
right-continuous path $\omega\dvtx \Real_+ \mapsto(0, \infty) \cup
\{
\triangle\}$, define $\zeta(\omega):= \inf\{t \in\Real_+| \omega
(t) = \triangle\}$. With the previous understanding, define $\Omega$
to be
the space of all right-continuous paths $\omega\dvtx \Real_+ \mapsto(0,
\infty) \cup\{\triangle\}$ such that $\omega(0) \in(0,\infty)$,
that are actually continuous on the interval $[0, \zeta(\omega))$ and
$\omega(t) = \triangle$ holds for all $t \in[\zeta(\omega), \infty)$.
Let $X$ denote the coordinate process on $\Omega$ and $\bF$ be the
right-continuous augmentation of the natural filtration of $X$; then
$\zeta$ becomes a stopping time on $(\Omega, \bF)$. 
Defining $\Omega$ as above is essential for ensuring that Assumption~\ref{ass: existence of prob} is valid; see the discussion on standard
systems and, more particularly, \cite{MR0368131}, Example 6.3.

Set $\beta(x) = 1 \vee x^2$ for $x \in(0, \infty)$. From the treatment
of \cite{MR1121940}, Section~5.5, there exists a probability $\prob$ on
$\F$ such that the coordinate process $X$ satisfies $\prob[X_0 = 1] =1
$ and has dynamics $\ud X_t = \beta(X_t) \,\ud W^\prob_t$, for $t \in[0,
\zeta)$, where 
$W^\prob$ is a standard Brownian motion under $\prob$. (In general,
$W^\prob$ is defined only up to time~$\zeta$.) In fact, $X$ is a strict
local martingale on $(\Omega, \bF, \prob)$ in the terminology of
\cite{MR1478722},
as follows from results in \cite{Del_Shi}. Using Feller's test for
explosions and the local martingale property, it is straightforward to
check that $\prob[\zeta\leq t, X_{\zeta-} = 0] = \prob[\zeta\leq
t] > 0$ holds for all $t \in(0, \infty)$. Let $\rho$
denote a time of overall maximum of $X$. By Proposition~\ref{prop: pair
for max}, it follows that $L = X \indic_{[\![ 0, \zeta[\![}$. In order
to characterise the probability $\qprob$ that $L$ induces as in
Assumption~\ref{ass: existence of prob}, note that, \emph{if} $L$ was
actually the density process of $\qprob$ with respect to $\prob$,
Girsanov's theorem would imply that the dynamics of $X$ under $\qprob$
are $\ud X_t = (\beta^2(X_t) / X_t) \,\ud t + \beta(X_t) \,\ud
W^\qprob
_t$ for $t \in[0, \zeta)$, with $W^\qprob$ being a standard Brownian
motion on $(\Omega, \bF, \qprob)$. Even though $L$ is not a
martingale on $(\Omega, \bF, \prob)$,
the treatment of \cite{MR1121940}, Section~5.5, implies that there
exists a probability $\qprob$ on $(\Omega, \F)$ such that the
coordinate process $X$ indeed satisfies $\qprob[X_0 = 1] = 1$ and $\ud
X_t = (\beta^2(X_t) / X_t) \,\ud t + \beta(X_t) \,\ud W^\qprob_t$ for
$t \in[0, \zeta)$, where $W^\qprob$ is a standard Brownian motion
under $\qprob$, in general defined until time $\zeta$. It is also clear
that $\qprob$ is exactly the probability that appears in Assumption~\ref
{ass: existence of prob}. Writing the formal dynamics under $\qprob$ of
$1/X$ on the stochastic interval $[\![ 0, \zeta[\![$, it is
straightforward to conclude that the law of $(1 / X_t)_{t \in[0, \zeta
)}$ under $\qprob$ is the same as the law of $(X_t)_{t \in[0, \zeta
)}$ under $\prob$. It follows that $\prob[\zeta\leq t, X_{\zeta-}
= \infty] = \prob[\zeta\leq t] > 0$ holds for all $t
\in(0, \infty)$. Coupled with the fact that $\prob[\zeta\leq t,
X_{\zeta-} = 0] = \prob[\zeta\leq t] > 0$ holds for all $t
\in
(0, \infty)$ that was established above, we conclude that neither
$\qprob\ll_{\F_t} \prob$ nor $\prob\ll_{\F_t} \qprob$ holds,
for any
$t \in(0, \infty)$.

The above example also illustrates that the filtration $\bF$ should not
be completed in any way by $\prob$, if $\qprob$ is to be defined. In
fact, let $\bF^\prob= (\F^\prob_t)_{t \in{\Real_+}}$ be \emph{any}
right-continuous filtration such that:
\begin{itemize}
\item$\bF\subseteq\bF^\prob$, and
\item if $B \subseteq\bigcup_{n \in\Natural} B_n$, where $B_n \in
\bigcup_{t
\in\Real_+} \F_t$ and $\prob[B_n] = 0$ holds for all $n \in
\Natural$, then
$B \in\F^\prob_0$.
\end{itemize}
(Note that we are \emph{not} asking that each $\F^\prob_t$, $t \in
{\Real_+}
$, contains all $\prob$-null sets of $\F_\infty$, but a weaker
condition that is tailored to avoid problems with singularities of
probabilities at infinity; see \cite{MR1906715} for the concept of such
\emph{natural}, as opposed to \emph{usual}, augmentations.) For any
$n \in\Natural$, $B_n:= \{\zeta\leq n, X_{\zeta-} = \infty\} \in
\F_n$
and $\prob[B_n] = 0$. In view of the assumptions on $\bF^\prob$,
$\{\zeta< \infty, X_{\zeta-} = \infty\} \in\F^\prob_0$. If
$\qprob$ could be defined, $\qprob|_{\F^\prob_{\eta_u}} \ll\prob
|_{\F
^\prob_{\eta_u}}$ would hold for $u \in[0, 1)$; in particular,
$\qprob
^\prob|_{\F^\prob_0} \ll\prob|_{\F^\prob_0}$. This is impossible:
indeed, we should have $\qprob[\zeta< \infty, X_{\zeta-} =
\infty
] = 1$, while it is true that $\prob[\zeta< \infty, X_{\zeta-} =
\infty] = 0$. Of course, since the filtration is \emph{not} enlarged in
order to include $\prob$-null sets, one can define $\qprob$ without problems.
\end{exa}

\subsection{Last-passage times of nonnegative continuous-path local
martingales vanishing at infinity} \label{subsec: last exit of local marts}

Let $M$ be a nonnegative local martingale on $(\Omega, \bF, \prob)$
with $M_0 = 1$,
$M$ having continuous paths and $\lim_{t \to\infty} M_t = 0$, all
holding $\prob$-a.s. In particular, and in the notation of
Section~\ref{subsec: sup of cont local marts}, $M \in\mathcal{L}_0$.
We fix $y \in\Real_+$
and define $\rho:= \sup\{t \in\Real_+ | M_t = y\}$, setting
$\rho= 0$ when the last set is empty. In words, $\rho$ is the last
passage time of $M$ at level $y$. In this case, it is straightforward that
\[
Z_t = \prob[\rho> t | \F_t] = \frac{M_t}{y}
\wedge1\qquad \mbox{for all } t \in{\Real_+}.
\]
(The set-inclusion $\{M > y\} \subseteq\{Z = 1\}$ certainly holds
modulo $\prob$; the fact that $Z = M / y$ holds on $\{M \leq y\}$
follows from Doob's maximal identity \cite{MR2247846} because $M$ has
$\prob$-a.s. continuous paths.)

Recall from Section~\ref{subsec: canon pair} that $Z = N - A$ holds for
an appropriate local martingale $N$ on $(\Omega, \bF, \prob)$. In
order to compute
$N$ and $A$ in the decomposition of~$Z$, information on the jumps of
$A$ is required. Since $A$ is the dual optional projection of $\indic_{
[\![ \rho, \infty[\![ }$ on $(\Omega, \bF, \prob)$, $\Delta A_\tau=
\prob[\rho= \tau| \F_{\tau}]$ holds for any finite stopping time
$\tau
$. Note that $A_0 = \prob[\rho= 0] = 1 - Z_0 = 0 \vee(1 - 1 / y)$.
Furthermore, on $\{\tau> 0, M_\tau\neq y\}$, it is clear that
$\prob[\rho= \tau| \F_\tau] = 0$ holds for any finite stopping
time $\tau$. Furthermore, $\prob[\rho\geq\tau| \F_\tau] =
1$ holds on $\{M_{\tau} = y\} \subseteq\{Z_\tau= 1\}$, which
implies that on $\{\tau> 0, M_\tau= y\}$ it holds that $\prob
[\rho
= \tau| \F_\tau] = 1 - \prob[\rho> \tau| \F_\tau] = 1 -
Z_\tau= 0$. We conclude that $\Delta A_\tau= 0$ on $\{\tau> 0\}$,
which implies that $A$ is a continuous-path process. It follows that $Z
= N - A$ coincides with the Doob--Mayer decomposition of $Z$, where $N$
is (necessarily) a continuous-path martingale with $N_0 = 1$. By the
Meyer--It\^o--Tanaka formula \cite{MR1037262}, Theorem IV.70, it holds
that $\ud N_t = (1/ y)\indic_{\{M_{t} \leq y\}} \,\ud M_t$ and $\ud A_t
= (1/2y) \,\ud\Lambda^M_t (y)$ for $t \in(0, \infty)$, where
$(\Lambda
^M_t (y))_{t \in\Real_+}$ denotes the semimartingale local time of $M$
at level $y$---see \cite{MR1037262}, page~216. A bit of algebra on
\eqref
{eq: real defn of K} gives
%
\begin{equation}
\label{eq: K for last exit} K = 1 - \biggl(1 \wedge\frac{1}{y} \biggr) \exp \biggl(-
\frac{1}{2 y} \Lambda^M (y) \biggr).
\end{equation}
Furthermore, since $\{M \leq y\} \subseteq\{y Z = M\}$, the
dynamics $\ud N_t = (1/ y)\*\indic_{\{M_{t} \leq y\}} \,\ud M_t$ for
$t \in\Real_+
$ and \eqref{eq: defn of L} give
%
\begin{equation}
\label{eq: L for last exit} \frac{\ud L_t}{L_{t}} = \indic_{\{M_{t} \leq y\}} \frac{\ud
M_t}{M_{t}}\qquad
\mbox{for } t \in[0, \zeta_0).
\end{equation}

\begin{rem}
If Assumption~\ref{ass: existence of prob} is valid, the dynamics in
\eqref{eq: L for last exit} suggest that the stochastic behaviour of
processes under $\qprob$ is like the one under $\prob$ when $M > y$;
furthermore, when $M \leq y$, the stochastic behaviour of processes
under $\qprob$ is like the one under the corresponding probability
$\qprob$ when the random time is the time of maximum of $M$, studied in
Section~\ref{subsec: sup of cont local marts}. The reader should also
check Example~\ref{exa: exchange} in Section~\ref{subsec:
exch-last-pass} for dynamics under $\qprob$ in a one-dimensional
diffusion setting.
\end{rem}

\begin{rem}\label{rem: last_exit_cont_mart}
Suppose that $y \in(0,1]$. In this case, $K = 1 -\break \exp(- (1 /2 y)\times
\Lambda^M (y))$, so that $\Delta K = 0$ up to a $\prob$-evanescent
set. By Proposition~\ref{prop: equiv for rand time avoiding all stop
times}, $K_\rho= K_\infty$ has the standard uniform distribution under
$\prob$. It follows that $\Lambda^M_\infty(y) = \Lambda^M_\rho(y)$
has the exponential distribution with rate parameter $2 y$ under $\prob
$. Also, note that in this case that the last exit time $\rho$ is
actually the time of maximum of $L$, which becomes apparent once one writes
\[
L = \frac{Z}{1 - K} = \biggl(\frac{M}{y} \wedge1 \biggr) \exp \biggl(
\frac{1}{2 y} \Lambda^M (y) \biggr)
\]
and use the facts that $\prob[M_\rho= y] = 1$ and $\prob [
\Lambda
^M_\rho(y) = \Lambda^M_\infty(y)  ] = 1$.
\end{rem}

\begin{exa}\label{exa: BM - last}
Recall the Brownian setting of Example~\ref{exa: BM - sup}. Suppose
that $x \in\Real$. Define
$\rho:= \sup\{t \in\Real_+| X_t = x\}$, where we set $\rho= 0$ when
the last set is empty. Recalling that $M = \exp(-2 \mu X)$, it holds
that $\rho:= \sup\{t \in\Real_+ | M_t = y\}$, where $y =
\exp
(- 2 \mu x)$. Furthermore, straightforward computations using a
combination of the two occupation-times formulas for $\Lambda^X$ and
$\Lambda^M$ imply that we can choose the local times in a way so that
$(1/ y) \Lambda^M (y) = - 2 \mu\Lambda^X(x)$. Therefore, equation
\eqref{eq: K for last exit} in this case reads $K = 1 - (1 \wedge\exp
(2 \mu x) ) \exp(\mu\Lambda^X (x))$. By Proposition~\ref{prop:
equiv for rand time avoiding all stop times}, it follows that $\Lambda
^X_\infty(x) = \Lambda^X_\rho(x)$ is such that $\prob[\Lambda
^X_\infty(x) = 0] = 1 - \exp(2 \mu x)$ when $x \in(0, \infty)$ and
$\prob[\Lambda^X_\infty(x) = 0] = 0$ when $x \in(- \infty, 0]$;
furthermore, given $\Lambda^X_\infty(x) > 0$, $\Lambda^X_\infty(x)$
has the exponential distribution with rate parameter $- \mu$ under
$\prob$.

Using Novikov's condition (\cite{MR1121940}, Section~3.5.D), it is
straightforward to check that the local martingale $L$ in \eqref{eq: L
for last exit} is an actual martingale. The extension theorem of
Daniell--Kolmogorov (\cite{MR1121940}, Section~2.2A), implies that
Assumption~\ref{ass: existence of prob} is valid in this case (modulo
the enlargement of the probability space in order to accommodate a
uniform random variable). It is straightforward to check that, under
$\qprob$, the process $X$ has dynamics $\ud X_t = \mu\sign(X_t - x)
\,\ud t + \ud W^\qprob_t$ for $t \in\Real_+$, where $\sign= \indic_{
( 0,\infty) } - \indic_{( - \infty, 0 ]}$ and $W^\qprob
$ is a
standard Brownian motion under $\qprob$. Dynamics like the ones of $X$
under $\qprob$ have been the object of study in previous literature;
see, for example, \cite{shiryaev:412} and \cite{MR1894767},
Section~5.2, page~96.\vadjust{\goodbreak}
\end{exa}

\section{Applications to financial mathematics} \label{sec: finance}

\subsection{Market behaviour up to the time of overall minimum of the
num\'eraire portfolio} \label{subsec: insider}

For the purposes of Section~\ref{subsec: insider}, we shall not be
needing Assumption~\ref{ass: existence of prob}; $(\Omega, \bF,
\prob)$ is taken to
be a filtered probability space, where $\bF$ actually satisfies the
usual conditions of right-continuity and augmentation by $\prob$-null
sets of $\F$. On $(\Omega, \bF, \prob)$, let $S = (S^i)_{i=1,\ldots, d}$ be a
sigma-bounded $d$-dimensional semi-martingale. (The condition of
sigma-boundedness is weaker than local boundedness of $S$---in fact, it
is equivalent to the existence of strictly positive and nonincreasing
predictable processes $\vartheta^i$ such that $\int_0^\cdot\vartheta
^i_t \,\ud S^i_t$ is a uniformly bounded process for each $i \in\{1,\ldots, d\}$. For the concepts of sigma-localisation and
sigma-martingales, the reader can refer to \cite{MR2013413}. The
concept of sigma-boundedness has also appeared in \cite{MR2260066}.)
For each $i \in\{1,\ldots, d\}$, $S^i$ represents the discounted,
with respect to some baseline security, price of a liquid asset in the
market. This baseline security should be thought as a locally riskless
asset. Starting with normalised unit capital, and investing according
to some $d$-dimensional, $\bF$-predictable and $S$-integrable strategy
$\vartheta$ (modelling the number of liquid assets held in the
portfolio), an economic agent's discounted wealth is given by
$X^{\vartheta} = 1 + \int_0^\cdot\vartheta_t \cdot\,\ud S_t$.
(Stochastic integrals with respect to $S$ are to be understood in the
sense of vector stochastic integration; see \cite{MR1943877}.) Define
$\X$ as the set of all processes $X^{\vartheta}$ in the previous
notation that remain nonnegative at all times.

\begin{ass} \label{ass: num}
In the above set-up, assume the following:
\begin{enumerate}[(1)]
\item[(1)]There exists $\Xhat\in\X$ with the following properties:
\begin{enumerate}[(a)]
\item[(a)]$X / \Xhat$ is a super-martingale for all $X \in\X$.
\item[(b)]$\Delta\Xhat\geq0$ up to $\prob$-evanescence. Furthermore,
with $\widehat{I}:= \inf_{t \in[0, \cdot]} \Xhat$, the
set-inclusion $\{ \Xhat_- = \widehat{I}_- \} \subseteq\{ \Delta
\Xhat
= 0\}$ holds up to $\prob$-evanescence.
\end{enumerate}
\item[(2)] There exists $X \in\X$ such that $\prob[\lim_{t \to\infty}
X_t = \infty] = 1$.
\end{enumerate}
\end{ass}

\begin{rem}
Condition (1) in Assumption~\ref{ass: num} is connected to market
viability, and in particular to \emph{absence of arbitrage of the first
kind}, that is, condition NA$_1$. (The market allows for arbitrage of
the first kind if there exists $T \in\Real_+$ and an $\F_T$-measurable
random variable $\xi$ with the properties $\prob[\xi\geq0] = 1$ and
$\prob[\xi> 0] > 0$, and such that for all $x > 0$ there exists $X
\in
x \X$, which may depend on $x$, satisfying $\prob[X_T \geq\xi] = 1$.)
Condition NA$_1$ is actually equivalent to the requirement that $\lim_{m \to\infty} \sup_{X \in\X} \prob[X_T > m] = 0$ holds for all
$T \in\Real_+$---see \cite{Kar_09_fin_add_ftap}, Proposition~1. It
then comes as a consequence of results in \cite{MR2335830} that absence
of arbitrage of the first kind is equivalent to existence of $\Xhat\in
\X$ such that $X / \Xhat$ is a super-martingale for all $X \in\X$,
which is exactly condition (1)(a). Condition (1)(b) in Assumption~\ref{ass:
num} additionally forces certain requirements which will enable use of
results from Section~\ref{subsec: sup of cont local marts} and are
crucial for the development below.
\end{rem}

Condition (1) of Assumption~\ref{ass: num} implies in particular that
$1 / \Xhat$ is a super-martingale on $(\Omega, \bF, \prob)$. The
next result refines
this observation.

\begin{lem} \label{lem: recip_num_lmart}
Under condition (1) of Assumption~\ref{ass: num}, $1 / \Xhat$ is a
local martingale on $(\Omega, \bF, \prob)$.
\end{lem}

\begin{pf}
Since both $\Xhat_- > 0$ and $\Xhat> 0$ hold, we have $\Xhat= 1 +
\int_0^\cdot\Xhat_{t-} (\varphi_t \cdot\ud S_t)$ for some
$d$-dimensional predictable and $S$-integrable process $\varphi$. A\break
straightforward application of \cite{MR2335830}, Lemma~3.4, shows that
$L:= 1 / \Xhat= 1 - \int_0^\cdot L_{t-} (\varphi_t \cdot\ud
\Shat_t)$, where
\[
\Shat:= S - \biggl[{}^\mathsf{c} S, \int_0^\cdot
\bigl(\varphi_t \cdot\ud {}^\mathsf{c} S _t
\bigr) \biggr] - \sum_{t \leq\cdot} \frac{\Delta\Xhat
_t}{\Xhat_t} \Delta
S_t,
\]
with ${}^\mathsf{c} S$ denoting the uniquely defined continuous local
martingale
part of $S$ (see, e.g., \cite{MR1943877}) and $[\cdot, \cdot]$ denotes
the operator returning the quadratic covariation of semi-martingales.
Since $L_- > 0$ and $L > 0$, $L$ is a local martingale if and only if
$\int_0^\cdot(\varphi_t \cdot\ud\Shat_t)$ is a local
martingale. The super-martingale property of $L$ already gives that
$\int_0^\cdot(\varphi_t \cdot\ud\Shat_t)$ is a local
sub-martingale. We shall show that $\int_0^\cdot(\varphi_t, \ud
\Shat_t)$ is also a local super-martingale. Since $2 \varphi\cdot
\Delta S = 2 (\Delta\Xhat/ \Xhat_-) \geq0$, the process $X'$ defined
implicitly via $X' = 1 + \int_0^\cdot X'_{t-} (2 \varphi_t \cdot\ud
S_t)$ is an element of $\X$ with $X' > 0$ and \mbox{$X'_- > 0$}.
Therefore, $X' / \Xhat$ is a nonnegative super-martingale. Again,
\cite{MR2335830}, Lemma~3.4, shows that $X' / \Xhat= 1 + \int_0^\cdot
(X'_{t-} / \Xhat_{t-}) (\varphi_t \cdot\ud\Shat_t)$. The
super-martingale property of $X' / \Xhat$ implies that $\int_0^\cdot
(\varphi_t \cdot\ud\Shat_t)$ is a local super-martingale, which
completes the argument.
\end{pf}

\begin{rem}
Lemma~\ref{lem: recip_num_lmart} above follows part of the proof of
\cite{Kar09}, Theorem~2.15. While the latter result really requires the
full force of condition (1) in Assumption~\ref{ass: num} in order to be
valid, the set-inclusion $\{ \Xhat_- = \widehat{I}_- \} \subseteq\{
\Delta\Xhat= 0\}$ was erroneously neglected in \cite{Kar09}, Theorem~2.15.
\end{rem}

Given condition (1)(a) in Assumption~\ref{ass: num}, the nonnegative
super-mar\-tin\-gale convergence theorem implies that condition (2)\vspace*{1pt} in
Assumption~\ref{ass: num} is actually equivalent to $\prob [ \lim_{t \to\infty} \Xhat_t = \infty ] = 1$. Let $L:= 1 / \Xhat$.
Since $L_0 = 1$ and Assumption~\ref{ass: num} implies that $L^*$ is
continuous and $\prob[L_\infty= 0] = 1$, Lemma~\ref{lem:
recip_num_lmart} and condition (1) of Assumption~\ref{ass: num} imply
that $L \in\mathcal{L}_0$, in the notation of Section~\ref{subsec:
sup of cont
local marts}. By Proposition~\ref{prop: pair for max}, it follows that
there exists a $\prob$-a.s. unique time $\rho$ of minimum of~$\Xhat$,
and that $(1 - 1/L^*, L)$ is the canonical representation pair
associated with~$\rho$. Let $\bG= (\G_t)_{t \in{\Real_+}}$ be the smallest
right-continuous filtration that contains $\bF$ and makes the random
variable $\widehat{I}_\infty=\inf_{t \in\Real_+} \Xhat_t$ be $\G
_0$-measurable. In this case, $\rho$ is $\prob$-a.s. equal to the first
time that $\Xhat$ equals $\widehat{I}_\infty$, which is a stopping time
on $(\Omega, \bG)$; since $\bF$ satisfies the usual conditions, we conclude
that $\rho$ is a stopping time on $(\Omega, \bG)$.

When $S$ consists of continuous-path semi-martingales, a version of the
next result appears in \cite{Kar12}, Theorem~1.4. The strengthened
result that is presented here has a short proof due to the
previously-built theory.

\begin{thmm} \label{thmm: main}
Under Assumption~\ref{ass: num} and the above notation, the
$d$-dimensional process $S^\rho= (S_{\rho\wedge t})_{t \in
\Real
_+}$ is a sigma-martingale on $(\Omega, \bG, \prob)$.
\end{thmm}

\begin{pf}
Let $X \in\X$. In the notation of Section~\ref{subsec: num property},
since $(X / \Xhat) \in\Sl$ and $\rho$ is a time of maximum of $L:=
1 / \Xhat$, which in particular avoids all stopping times in view of
Proposition~\ref{prop: pair for max}, it follows that $\expec_\prob
 [
X_\rho/ \Xhat_\rho| K_\rho ] \leq1 / \Xhat_\rho$. Since
$K_\rho= 1 - 1 / \Xhat_\rho$, the last equality translates to
$\expec_\prob
[X_\rho| K_\rho] \leq1$; in other words, $\expec_\prob[X_\rho
f(K_\rho)] \leq\expec_\prob[f(K_\rho)]$ is valid for all $X \in
\X$
and Borel-measurable $f\dvtx[0, 1)\mapsto\Real_+$. Now, fix $t_1 \in
\Real
_+$, $t_2 \in(t_1, \infty)$, $A \in\F_{t_1}$ and $X \in\X$ with $X
\geq1/2$. Let $\vartheta$ be so that $X = 1 + \int_0^\cdot\vartheta_t
\cdot\ud S_t$, and define $\vartheta':= (1 / X_{t_1}) \indic_A
\indic_{]\!] t_1, t_2 ]\!]} \vartheta$ and $X':= 1 + \int_0^\cdot
\vartheta'_t \cdot\ud S_t$. It is straightforward\vspace*{1pt} to check that $X'
\in\X$ and that $X'_\rho= \indic_{\Omega\setminus A} + (X^\rho
_{t_2} / X^\rho_{t_1}) \indic_A$. Therefore, the inequality $\expec
_\prob
[X'_\rho f(K_\rho)] \leq\expec_\prob[f(K_\rho)]$ gives $\expec
_\prob
[(X^\rho_{t_2} / X^\rho_{t_1}) f(K_\rho) \indic_A] \leq
\expec_\prob
[f(K_\rho) \indic_A]$. Defining $\G^0_t = \F_{t} \vee\sigma
(K_\rho)$
for all $t \in\Real_+$ and ranging $A \in\F_{t_1}$, we obtain that
$\expec_\prob
[X^\rho_{t_2} | \G^0_{t_1}] \leq X^\rho_{t_1}$ holds for all
$t_1 \in\Real_+$, $t_2 \in(t_1, \infty)$ and $X \in\X$ with $X
\geq
1/2$. By definition of the filtration $\bG$, $\G_{t_1} = \bigcap_{t >
t_1} \G^0_t$ holds; then, the conditional version of Fatou's lemma
gives that $\expec_\prob[X^\rho_{t_2} | \G_{t_1}] \leq X^\rho
_{t_1}$ holds for all $t_1 \in\Real_+$, $t_2 \in(t_1, \infty)$ and $X
\in\X$ with $X \geq1/2$. Ranging $t_1 \in\Real_+$ and $t_2 \in(t_1,
\infty)$, we obtain that $X^\rho$ is a super-martingale on $(\Omega,
\bG, \prob)$
for all $X \in\X$ with $X \geq1/2$.

For each $i \in\{1,\ldots, d\}$ pick a strictly positive and
nonincreasing predictable process $\vartheta^i$ such that $|\int_0^\cdot\vartheta^i_t \,\ud S^i_t| \leq1/2$ identically holds. In this
case, both processes $1 + \int_0^\cdot\vartheta^i_t \,\ud S^i_t$ and $1
- \int_0^\cdot\vartheta^i_t \,\ud S^i_t$ are elements of $\X$ and
bounded below by $1/2$. It follows that $\int_0^{\rho\wedge\cdot}
\vartheta^i_t \,\ud S^i_t$ is both a super-martingale and a
sub-martingale on $(\Omega, \bG, \prob)$, which means that it is a
martingale on
$(\Omega, \bG, \prob)$. Since $\vartheta^i$ is strictly positive,
this implies that
$(S^i_{\rho\wedge t})_{t \in\Real_+}$ is a sigma-martingale on
$(\Omega, \bG, \prob)$ for
all $i \in\{1,\ldots, d\}$.
\end{pf}

The importance of Theorem~\ref{thmm: main} lies in the following
observation: with the ``insider information'' flow $\bG$, investing in
the risky assets before time $\rho$ gives the same instantaneous return
as the (locally) riskless asset, but entails (locally) higher risk;
therefore, before $\rho$ an insider would not be willing to take any
position on the risky assets. In a sense, Theorem~\ref{thmm: main}
endows $\Xhat$ the quality of an index of market status. Extensive
discussion on this and further remarks can be found in \cite{Kar12}.

\subsection{Valuation of exchange options and last-passage times}
\label
{subsec: exch-last-pass}

In recent literature, there has been considerable interest in
representations of the value of plain vanilla options in terms of last
passage times---in fact, the monograph \cite{MR2582990} contains much
of this development. Last-passage times for continuous local
martingales that vanish at infinity were considered in Section~\ref
{subsec: last exit of local marts}; that discussion will be used here
to provide a further representation for the value of exchange options.

On $(\Omega, \bF, \prob)$, let $S^0$ and $S^1$ be two nonnegative
continuous-path
semi-martingales. The process $S^0$ satisfies $S^0_0 = 1$ and $\prob
[\inf_{t \in[0, T]} S^0_t > 0] = 1$ for all $T \in\Real_+$, and
should be considered as a baseline security. Set $R:= S^1 / S^0$ to
denote the ``exchange rate,'' that is, the price process $S^1$
denominated in units of the baseline asset with price process $S^0$.

In the above market, consider an option to exchange at time $T \in
\Real
_+$ a unit of a security with price process $S^1$ for $\kappa$ units of
the baseline security $S^0$. The option will be valid at time $T$ only
if the event $\{\sigma\leq T\}$ has occurred, where $\sigma$ is a
stopping time on $(\Omega, \bF)$. For example, one could take $\sigma
= \inf
\{t \in\Real_+| R_t > \lambda\}$ for some $\lambda> \kappa$, in which
case the security is really an ``up-and-in'' exchange option. For a
plain vanilla exchange option, one may set $\sigma= 0$.

Given that $\prob$ is the valuation measure and that discounting is
done using the baseline security, as is typically the case, the value
of a European exchange option of the aforementioned type, to be
exercised at time $T \in\Real_+$, is $\mathsf{EE}_T= \expec_\prob
[(\kappa- R_T)_+ \indic_{\{\sigma\leq T\}}]$. Note that $\prob$ is an
equivalent local martingale measure for $R$, which means that $R$ is a
nonnegative local martingale on $(\Omega, \bF, \prob)$.

\begin{rem}
In fact, the valuation formula for the European option is valid also
for the value of the corresponding American option. In order to see
this, let $\Stop_{[0, T]}$ be the class of all stopping times $\tau$ on
$(\Omega, \bF)
$ satisfying $0 \leq\tau\leq T$. Using $\prob$ as valuation measure,
an American option of the previous type has value $\ame:= \sup_{\tau
\in\Stop_{[0, T]}} \expec_\prob[(\kappa- R_\tau)_+ \indic_{\{
\sigma\leq\tau\}}]$. Given that $R$ is a nonnegative local
martingale on
$(\Omega, \bF, \prob)
$, thus a super-martingale on $(\Omega, \bF, \prob)$, it is
straightforward that the
process $((\kappa- R_t)_+)_{t \in\Real_+}$ is a sub-martingale on
$(\Omega, \bF, \prob)
$. Then, for any $\tau\in\Stop_{[0, T]}$ it holds that
\[
\expec_\prob \bigl[(\kappa- R_T)_+ \indic_{\{\sigma\leq T\}}
| \F _\tau \bigr] \geq\expec_\prob \bigl[(\kappa-
R_T)_+ \indic_{\{\sigma\leq
\tau\}} | \F_\tau \bigr] \geq(
\kappa- R_\tau)_+ \indic_{\{\sigma\leq
\tau\}},
\]
which readily gives
\[
\ame= \sup_{\tau\in\Stop_{[0, T]}} \expec_\prob \bigl[(\kappa-
R_\tau )_+ \indic_{\{\sigma\leq\tau\}} \bigr] = \expec_\prob
\bigl[( \kappa- R_T)_+ \indic_{\{\sigma\leq T\}} \bigr] =
\mathsf{EE}_T.
\]
\end{rem}

For $\kappa\in\Real_+$, define the random time $\rho:= \sup\{t \in
\Real_+ | R_t = \kappa\}$, where we set $\rho= 0$ if the last
set is empty. Under the force of Assumption~\ref{ass: existence of
prob}, denote by $\qprob$ the probability corresponding to $\rho$.

\begin{prop} \label{prop: exchange}
In the above set-up, suppose that $\prob[\lim_{t \to\infty} R_t =
0] = 1$ and that the validity of Assumption~\ref{ass: existence of
prob} is in force for the random time $\rho$. Then it holds that
%
\begin{eqnarray}
\label{eq: eur_rep} \mathsf{EE}_T
&=& \kappa\prob[\rho\wedge\sigma\leq T]
\nonumber
\\[-8pt]
\\[-8pt]
\nonumber
&=&
\kappa\prob [\sigma\leq T] - \kappa(1\wedge\kappa) \expec_\qprob \biggl[
\exp \biggl(- \frac{\kappa}{2} \Lambda^R_{T} (\kappa)
\biggr) \indic_{\{\sigma\leq T\}
} \biggr].
\end{eqnarray}
\end{prop}

\begin{pf}
Under the validity of $\prob[\lim_{t \to\infty} R_t = 0] = 1$,
the equality $(\kappa- R_T)_+ = \kappa\prob[\rho\leq T | \F_T]$
holds in view of \cite{MR2582990}, Theorem~2.5; then the
first equality in \eqref{eq: eur_rep} follows from the fact that $\{
\sigma\leq T\} \in\F_T$. For the second equality in \eqref{eq:
eur_rep}, note that, in view of \eqref{eq: K for last exit}, the
process $K$ in the canonical representation pair of $\rho$ on $(\Omega, \bF, \prob)$
is such that $1 - K = (1 \wedge\kappa) \exp(- (\kappa/ 2) \Lambda
^R (\kappa))$. By Remark~\ref{rem: tail_prob}, and since
$\{\sigma\leq T\} \in\F_T$,
\[
\prob[\rho\wedge\sigma\leq T] = \prob[\sigma\leq T] - \prob[\sigma\leq T, \rho>
T] = \prob[\sigma\leq T] - \expec_\qprob \bigl[(1 - K_T)
\indic_{\{\sigma\leq T\}} \bigr],
\]
which completes the proof.
\end{pf}

\begin{exa} \label{exa: exchange}
We present here an example where the ``exchange rate'' process $R$
behaves as a one-dimensional diffusion under $\prob$. Exact modelling
of $S^0$ and $S^1$ is not necessary.

The filtered measurable space will be the exact one considered in
Example~\ref{exa: need_for_no_compl}, where the reader is referred to
for all the details. Recall that $X$ denotes the coordinate process and
$\bF$ be the right-continuous augmentation of the natural filtration of
$X$. The sigma-algebra $\F$ is taken to be equal to $\F_\infty$. Note
that this set-up is essential for ensuring that Assumption~\ref{ass:
existence of prob} is valid (modulo the enlargement discussed in Remark~\ref{rem: no loss of gen the unif} in order to accommodate for an
independent uniform random variable).

Fix a function $\beta\dvtx (0, \infty) \mapsto(0, \infty)$ such
that $1 /
\beta^2$ is locally integrable on $(0, \infty)$. From the treatment of
\cite{MR1121940}, Section~5.5, for any $x_0 \in\Real_+$ there exists a
probability $\prob$ on $\F$ (which coincides with the Borel
sigma-algebra on $\Omega$) such that $\prob[X_0 = x_0] =1$, and $X$
has dynamics
\[
\frac{\ud X_t}{X_t} = \beta(X_t) \,\ud W^\prob_t\qquad
\mbox{for } t \in [0, \zeta),
\]
where recall that $\zeta:= \inf\{t \in\Real_+| X_t = \triangle\}$,
and $W^\prob$ is a standard Brownian motion (defined only up to time
$\zeta$) under $\prob$. Due to the nonnegative local martingale
convergence theorem and the fact that $\beta\dvtx (0, \infty)
\mapsto(0,
\infty)$ is such that $1 / \beta^2$ is locally integrable on $(0,
\infty
)$, it follows in straightforward way that $\prob[X_{\zeta-} = 0]
= 1$. Letting $R:= X \indic_{[\![ 0, \zeta[\![}$, note that the
assumptions of Proposition~\ref{prop: exchange} are satisfied.

Regarding the probability $\qprob$, \eqref{eq: L for last exit} implies
that the local martingale $L$ on $(\Omega, \bF, \prob)$ in the canonical
representation pair of $\rho$ is such that $\ud L_t / L_t = \indic
_{\{X_t \leq\kappa\}} (\ud X_t / X_t ) = \indic_{\{X_t \leq\kappa\}}
\beta(X_t) \,\ud W^{\prob}_t$, for $t \in[0, \zeta)$. Using Girsanov's
theorem, it is straightforward to then check that
%
\begin{equation}
\label{eq: exch_dyn_q} \frac{\ud X_t}{X_t} = \beta^2 (X_t)
\indic_{\{X_t \leq\kappa\}} \,\ud t + \beta(X_t) \,\ud W^{\qprob}_t\qquad
\mbox{for } t \in[0, \zeta),
\end{equation}
where $W^{\qprob}$ is a standard Brownian motion under $\qprob$. (Even
though $L$ may fail to be a true martingale on $(\Omega, \bF, \prob
)$, one infers
the existence of the probability $\qprob$ on $(\Omega, \F)$ such
that the dynamics of $X$ are given by \eqref{eq: exch_dyn_q} using
knowledge of weak solutions of stochastic differential equations with
possible explosions from the treatment of \cite{MR1121940},
Section~5.5.) By employing Feller's test for explosions,
it can be
easily seen that $X$ under $\qprob$ does not explode, that is, does not
exit $(0, \infty)$ in finite time; that is, $R = X$ under $\qprob$. In
fact, by calculating the scale function of~$X$, one may conclude that
$R = X$ becomes a recurrent Markov process under $\qprob$.
\end{exa}

\section{Time of maximum and last-passage times of Brownian motion with
drift over finite time-intervals} \label{sec: BM on finite interv}

\subsection{Set-up}

For the purposes of Section~\ref{sec: BM on finite interv}, $T \in
\Real
_+$ will be fixed. Define $\Omega$ as the canonical path-space of
continuous functions from $[0, T)$ to $\Real$. Call $X$ the coordinate
process, let $\bF= (\F_t)_{t \in[0, T)}$ be the right-continuous
augmentation of the natural filtration of $X$, and set $\F= \bigvee_{t
\in[0, T)} \F_t$.

\begin{rem}\label{rem: BM_fin_q_exist}
It is important to note that the canonical space of processes with
time-index $[0, T)$, as opposed to $[0, T]$, is considered here. As
will become clear, it is in this setting that we can ensure later the
validity of Assumption~\ref{ass: existence of prob} (modulo the
enlargement of the space in order to accommodate a random variable with
the uniform law and independent of $\F_\infty$, as discussed in Remark~\ref{rem: no loss of gen the unif}).
\end{rem}

Fix $\mu\in\Real$. On $(\Omega, \F)$, let $\prob$ be the probability
under which $X$ is a Brownian motion with drift $\mu$ and unit
diffusion coefficient. In the rest of Section~\ref{sec: BM on finite
interv}, and using the previously-developed theory, we discuss the
behaviour of $X$ up to the time of maximum and last-passage times of
$X$. We shall calculate the canonical associated pair $(K, L)$ in each
case, and via $L$ we shall describe the dynamics of $X$ under~$\qprob$
(generated by $L$). In view of Section~\ref{sec: equal in law}, this
gives a complete characterisation of the stochastic behaviour of
optional processes up to the random times that are considered.

\subsection{Time of maximum} \label{subsec: time_max_BM_fin}

Define $\rho:= \sup \{ t \in[0, T)| X_t = \sup_{s \in[0, T)}
X_s  \}$, where by convention one sets $\rho= T$ if the previous
set is empty.

In the sequel, we shall make use of the following functions, related to
the standard normal law:
\[
\oPhi(x) = \int_x^\infty\phi(y) \,\ud y\qquad
\mbox{where } \phi (x) = \frac{1}{\sqrt{2 \pi}} \exp \biggl(- \frac{x^2}{2}
\biggr), \mbox { for } x \in\Real.
\]
Define the function $F_\mu\dvtx (0, \infty) \times\Real_+ \mapsto
[0,1]$ via
%
\begin{eqnarray}
\label{eq: F_BM_fin} F_\mu(\tau, z)&:=& \exp(2 \mu z) \oPhi \biggl(
\frac{z + \mu\tau}{\sqrt
{\tau}} \biggr) + \oPhi \biggl(\frac{z - \mu\tau}{\sqrt{\tau}} \biggr)
\nonumber
\\[-8pt]
\\[-8pt]
\nonumber
&=& \int
_0^\tau \biggl(\frac{z}{\sqrt{2 \pi s^{3}}} \exp \biggl(-
\frac{(z - \mu
s)^2}{2 s } \biggr) \biggr) \,\ud s,
\end{eqnarray}
for $(\tau,z) \in(0, \infty) \times\Real_+$. The second equality
follows upon differentiation of the defining quantity giving $F_\mu$
with respect to the temporal variable. The fact that $F_\mu$ is $[0,
1]$-valued follows from the second representation, since the quantity
inside the integral is the density of the first hitting time of the
level $z$ for Brownian motion with drift $\mu$; see \cite{MR1121940}, page~197,
equation (5.12). By this last fact and the Markovian
property of Brownian motion, it is straightforward that
\[
Z_t =\prob[\rho> t | \F_t] = F_\mu
\bigl(T-t, X^\up_t - X_t \bigr)\qquad \mbox{for }
t \in[0, T),
\]
where recall that $X^\up= \sup_{t \in[0, \cdot]} X$. In preparation
for the formulas below, note that
%
\begin{eqnarray}
\label{eq: dF_BM_fin} \frac{\partial F_\mu}{\partial z} (\tau, z) = 2 \mu\exp(2 \mu z) \oPhi \biggl(
\frac{z + \mu\tau}{\sqrt{\tau}} \biggr) - \frac{2}{\sqrt
{\tau}} \phi \biggl(\frac{z - \mu\tau}{\sqrt{\tau}}
\biggr)
\nonumber
\\[-8pt]
\\[-8pt]
\eqntext{\mbox{for } (\tau,z) \in(0, \infty) \times\Real_+,}
\end{eqnarray}
where the fact that $\exp(2 \mu z) \phi(z / \sqrt{\tau} + \mu\sqrt
{\tau}) = \phi(z / \sqrt{\tau} - \mu\sqrt{\tau})$ for
$(\tau,z) \in(0, \infty) \times\Real_+$ holds was used in the above
calculation. Define also the function $f_\mu\dvtx (0, \infty)
\mapsto
\Real
$ via
\[
f_\mu(\tau):= - \frac{\partial F_\mu}{\partial z} (\tau, 0) = \frac
{1}{\sqrt{2 \pi\tau}}
\exp \biggl(- \frac{\mu^2 \tau}{2} \biggr)
- 2 \mu\oPhi (\mu\sqrt{\tau})\qquad \mbox{for
} \tau\in(0, \infty).
\]
Upon simple differentiation, it is easy to check that the function
$f_\mu$ is decreasing in $\tau\in(0, \infty)$. As $\lim_{\tau\to
\infty} f_\mu(\tau) = \max\{0, - 2 \mu\} \in\Real_+$, $f_\mu
$ is
nonnegative.

Since $Z$ has continuous paths and all martingales on $(\Omega, \bF,
\prob)$ have
continuous paths as well, it follows that $A$ is the continuous
nondecreasing process appearing in the additive Doob--Meyer
decomposition of $-Z$. In view of Proposition~\ref{prop: equiv for rand
time avoiding all stop times}, $\rho$ avoids all stopping times on
$(\Omega, \bF, \prob)$. A simple use of It\^o's formula gives, after
some term
cancellations, that
%
\begin{eqnarray}
\label{eq: Z_dyn_BM_fin_max} \ud Z_t = - \frac{\partial F_\mu}{\partial z} \bigl(T-t,
X^\up_t - X_t \bigr) \,\ud(X_t -
\mu t) - f_\mu(T - t) \,\ud X^\up_t
\nonumber
\\[-8pt]
\\[-8pt]
\eqntext{\mbox{for }
t \in[0, T).}
\end{eqnarray}
In particular, it holds that $A = \int_0^{\cdot} f_\mu(T - t) \,\ud
X^\up
_t$. From \eqref{eq: real defn of K}, it then follows that
%
\begin{equation}
\label{eq: K_BM_fin_max} K_t = 1 - \exp \biggl(- \int_0^t
f_\mu(T - s) \,\ud X^\up_s \biggr)\qquad \mbox {for
} t \in[0, T).
\end{equation}
Using the equality $L = Z / (1 - K)$, it follows that
%
\begin{equation}\qquad
\label{eq: L_BM_fin_max} L_t = F_\mu \bigl(T-t,
X^\up_t - X_t \bigr) \exp \biggl(\int
_0^t f_\mu(T - s) \,\ud
X^\up_s \biggr) \qquad\mbox{for } t \in[0, T).
\end{equation}
The next result ensures that Assumption~\ref{ass: existence of prob}
will be valid in this setting.

\begin{lem} \label{lem: Q_exist_BM_fim_max}
For all $t \in[0, T)$, it holds that $\expec_\prob[L_t] = 1$.
\end{lem}

\begin{pf}
Since $(L_t)_{t \in[0, T)}$ is a nonnegative local martingale on
$(\Omega, \bF, \prob)$ with $L_0 = 1$, $\expec_\prob[L_t] = 1$ for
all $t \in[0, T)$
will follow if $\expec_\prob[L^*_t] < \infty$ for all $t \in[0, T)$ is
established. Given that the function $F_\mu$ is a $[0,1]$-valued and
that the function $f_\mu$ is decreasing, \eqref{eq: L_BM_fin_max}
implies that $L^*_t \leq\exp ( f_\mu(T - t) X^\up_t  )$ holds
for all $t \in[0, T)$. Therefore, $\expec_\prob[L^*_t] < \infty$
for all
$t \in[0, T)$ will follow if it is established that $\expec_\prob
 [
\exp
( a X^\up_t )  ] < \infty$ holds for all $a \in\Real$ and $t
\in
\Real_+$. To see this, note first that in view of Girsanov's theorem
and H\"older's inequality, one may assume that $\mu= 0$. Then the
claim follows because, for $\mu= 0$, the law of $X^\up_t$ under
$\prob
$ is the same as the law of $|X_t|$ under $\prob$, and all exponential
moments of the latter law are finite.
\end{pf}

By Lemma~\ref{lem: Q_exist_BM_fim_max} and the extension theorem of
Daniell--Kolmogorov \cite{MR1121940}, Section~2.2A, there exists a
probability $\qprob$ on $(\Omega, \F)$ such that $L_t$ is the
density of $\qprob$ with respect to $\prob$ on $\F_t$ for all $t \in
[0, T)$. (It is exactly here that the point of Remark~\ref{rem:
BM_fin_q_exist} becomes relevant.) It follows either from \eqref{eq:
Z_dyn_BM_fin_max} of from \eqref{eq: L_BM_fin_max} that the dynamics of
$L$ are
\[
\frac{\ud L_t}{L_t} = - \frac{ (\partial F_\mu/ \partial z)
(T-t, X^\up_t - X_t) }{F_\mu(T-t, X^\up_t - X_t)} \,\ud(X_t - \mu t)\qquad
\mbox{for } t \in[0, T).
\]
A straightforward application of Girsanov's theorem imply that, under
$\qprob$, the dynamics of $X$ are
%
\begin{equation}
\label{eq: weird dynamics BM} \,\ud X_t = G_\mu \bigl(T-t,
X^\up_t - X_t \bigr) \,\ud t + \ud
W^\qprob_t\qquad \mbox{for } t \in[0, T),
\end{equation}
where $W^\qprob$ is a standard Brownian motion on $(\Omega, \bF,
\qprob)$ and
$G_\mu\dvtx
(0, \infty) \times\Real_+ \mapsto\Real$ is a function satisfying
$G_\mu(\tau, z) = \mu- (\partial F_\mu/ \partial z) (\tau,
z) /
F_\mu(\tau, z)$ for $(\tau,z) \in(0, \infty) \times\Real_+$. A
use of
\eqref{eq: dF_BM_fin} gives
%
\begin{eqnarray}
\label{eq: G_BM_fin} &&G_\mu(\tau, z)\nonumber\\
 &&\qquad= \mu+ \frac{ (2 / \sqrt{\tau}) \phi(z / \sqrt
{\tau} - \mu\sqrt{\tau} ) - 2 \mu\exp(2 \mu z) \oPhi
(z / \sqrt{\tau} + \mu\sqrt{\tau} )} {\oPhi(z / \sqrt{\tau} -
\mu\sqrt{\tau} ) + \exp(2 \mu z) \oPhi(z / \sqrt{\tau} + \mu
\sqrt{\tau} )}
\\
\eqntext{\mbox{for }
(\tau,z) \in(0, \infty) \times \Real_+.}
\end{eqnarray}

\begin{rem}
When $\mu\in(- \infty, 0)$, it is straightforward to calculate\break  $\lim_{\tau\to\infty} F_{\mu} (\tau, z) = \exp(2 \mu z)$ and $\lim_{\tau
\to\infty} G_{\mu} (\tau, z) = - \mu$ for all $z \in\Real_+$, as well
as $\lim_{\tau\to\infty} f_{\mu} (\tau, z) = - 2 \mu$. Formally
plugging these long-run limits in \eqref{eq: K_BM_fin_max}, \eqref{eq:
L_BM_fin_max} and \eqref{eq: weird dynamics BM}, the set-up and results
of Example~\ref{exa: BM - sup} are recovered.
\end{rem}

\begin{rem} \label{rem: BM_fin_hor_case_zero}
When $\mu= 0$, previous formulas simplify significantly. In this case,
$F_0(\tau, z) = 2 \oPhi(z / \sqrt{\tau})$ for $(\tau,z) \in(0,
\infty) \times\Real_+$, $f_0(\tau) = 1 / \sqrt{2 \pi\tau}$ for
$\tau
\in(0, \infty)$, and the function $G_0$ appearing in the dynamics
\eqref{eq: weird dynamics BM} is given by $G_0(\tau, z) = (1 / \sqrt
{\tau})(\phi(z / \sqrt{\tau} ) / \oPhi(z / \sqrt{\tau}) )$, for
$(\tau,z) \in(0, \infty) \times\Real_+$. Upon
differentiation, it is straightforward to check that $(0, \infty)
\times\Real_+ \ni(\tau, z) \mapsto G_0(\tau, z)$ is decreasing in
$\tau$ and increasing in $z$. This is a very plausible behaviour:
recalling the dynamics~\eqref{eq: weird dynamics BM} under $\qprob$,
one would expect the drift to increase both when $X$ is moving away
from its maximum and when the ``time to maturity'' $\tau= T - t$ is
getting shorter.
\end{rem}

It is conjectured that the function $(0, \infty) \times(0, \infty)
\ni
(\tau, z) \mapsto G_\mu(\tau, z)$ is decreasing in $\tau$ and
increasing in $z$ for all $\mu\in\Real$---this was discussed for the
case $\mu= 0$ in Remark~\ref{rem: BM_fin_hor_case_zero}. However, the
calculations toward proving such a statement for all $\mu\in\Real$
seem quite tedious. Proposition~\ref{prop: BM_on_finite_int} that
follows gives important information on $G_\mu$ for arbitrary $\mu\in
\Real$.

\begin{prop} \label{prop: BM_on_finite_int}
The function $G_\mu$ is $\Real_+$-valued and such that\break  $\liminf_{\tau
\downarrow0} (\inf_{z \in[w, \infty)} (\tau G_\mu(\tau, z) ))
\geq w$ holds for all $w \in(0, \infty)$. In parti-\break cular,~it
follows that $X$ is a local sub-martingale on $(\Omega, \bF, \qprob
)$ and that\break
$\qprob [ \liminf_{t \to T} (X^\up_t - X_t) = 0  ] = 1$.
\end{prop}

\begin{pf} 
Let $c \in\Real$ and $d \in\Real$. A simple change of variables
implies that
\begin{eqnarray*}
\exp(2cd) \oPhi(c+d) &=& \int_{c+d}^{\infty} \exp
\biggl(2 c d - \frac
{x^2}{2} \biggr) \frac{\ud x}{\sqrt{2 \pi}} \\
&= &\int
_{d -c}^{\infty} \exp \biggl(2 c d - \frac{(x + 2 c)^2}{2}
\biggr) \frac{\ud x}{\sqrt{2 \pi}}
\\
&=& \int_{d -c}^{\infty} \exp \bigl(2 c (d - c- x)
\bigr) \exp \biggl(- \frac
{x^2}{2} \biggr) \frac{\ud x}{\sqrt{2 \pi}}.
\end{eqnarray*}
When $x \geq d - c$, it holds that $c \exp(2 c (d - c- x)) \leq
c$, for any $c \in\Real$. Therefore, from the equalities above we
obtain $c \exp(2cd) \oPhi(c+d) \leq c \oPhi(d - c)$. Applying the
previous inequality above with $c = \mu\sqrt{\tau}$ and $d = z /
\sqrt
{\tau}$, it follows that $\mu\oPhi(z / \sqrt{\tau} - \mu\sqrt
{\tau}) - \mu\exp(2 \mu z) \oPhi(z / \sqrt{\tau} + \mu\sqrt
{\tau} ) \geq0$ for all $(\tau,z) \in(0, \infty) \times
\Real
_+$. By \eqref{eq: G_BM_fin}, we obtain
%
\begin{eqnarray}
\label{eq: G_estimate} G_\mu(\tau, z) \geq\frac{(2 / \sqrt{\tau}) \phi(z / \sqrt{\tau
} - \mu\sqrt{\tau} )} {\oPhi(z / \sqrt{\tau} - \mu\sqrt{\tau
} ) + \exp(2 \mu z) \oPhi(z / \sqrt{\tau} + \mu\sqrt{\tau} )}
\nonumber
\\[-8pt]
\\[-8pt]
\eqntext{\mbox{for all
} (\tau,z) \in(0, \infty) \times \Real_+,}
\end{eqnarray}
from which it immediately follows that $G_\mu$ is a nonnegative
function. The fact that $X$ is a local sub-martingale in $(\Omega, \bF, \qprob)$ then
follows from the dynamics \eqref{eq: weird dynamics BM}.

Continuing, fix $w \in(0, \infty)$. Using the uniform estimates $1 -
1/x^2 \leq x \oPhi( x) /\break   \phi(x) \leq1$, valid for $x \in(0, \infty)$
(see, e.g., \cite{MR2722836}, Theorem~1.2.3, page~11), and the
fact that the equality $\exp(2 \mu z) \phi(z / \sqrt{\tau} + \mu
\sqrt{\tau}) = \phi(z / \sqrt{\tau} - \mu\sqrt{\tau})$
holds for all $(\tau,z) \in(0, \infty) \times\Real_+$, we obtain that
\[
\lim_{\tau\downarrow0} \biggl(\inf_{z \geq w} \biggl(
\frac{ 2 \sqrt{\tau}
\phi(z / \sqrt{\tau} - \mu\sqrt{\tau} ) }{\oPhi(z / \sqrt{\tau
} - \mu\sqrt{\tau} ) + \exp(2 \mu z) \Phi(z / \sqrt{\tau} + \mu
\sqrt{\tau} )} \biggr) \biggr) = w.
\]
Therefore, \eqref{eq: G_estimate} gives $\liminf_{\tau\downarrow0}
(\inf_{z \geq w} (\tau G(\tau, z) )) \geq w$ for all $w
\in
(0, \infty)$.\break According to this fact and the dynamics given in \eqref
{eq: weird dynamics BM}, on the event\break $ \{ \liminf_{t \to T}
(X^\up
_t - X_t) > 0  \}$ one would obtain $\lim_{t \to T} X_t = \infty$
under $\qprob$---indeed, the drift term in the dynamics \eqref{eq:
weird dynamics BM} would dominate (up to a strictly positive random
variable) the quantity $1/(T - t)$ when $t$ approaches $T$, implying
that the behaviour of $X$ itself near $T$ would be explosive. However,
in that case $\lim_{t \to T} (X^\up_t - X_t) = 0$ would hold on $
\{
\liminf_{t \to T} (X^\up_t - X_t) > 0  \}$ under $\qprob$, since
$X_t < \infty$ holds for all $t \in[0, T)$. We conclude that $\qprob
[ \liminf_{t \to T} (X^\up_t - X_t) = 0  ] = 1$.
\end{pf}

\begin{rem}
The fact that $\qprob [ \liminf_{t \to T} (X^\up_t - X_t) = 0
 ]
= 1$ is the equivalent of $\qprob[\rho= T] = 1$ that was obtained in
the finite-horizon discrete-time analogue discussed in Example~\ref
{exa: sup in discrete time}. However, in contrast to Example~\ref{exa:
sup in discrete time}, the fact that $\prob [ \lim_{t \to T}
(X^\up
_t - X_t) > 0  ] = 1$ implies that in the present setting $\prob$
and $\qprob$ are singular probabilities on $\F$. (Note also that
$\prob
 [ \liminf_{t \to T} (X^\up_t - X_t) > 0  ] = 1$ implies
$\prob
 [ \lim_{t \to T} L_t = 0  ] = 1$, which directly shows the
singularity of $\prob$ and $\qprob$ on $\F$.)
\end{rem}

\subsection{Last-passage times} Fix $x \in\Real$ and define $\rho:=
\sup\{t \in[0, T)| X_t = x\}$, where one sets $\rho= 0$ if the
previous set is empty. Recalling the definition of the function $F_\mu$
from \eqref{eq: F_BM_fin}, it is straightforward to compute
%
\begin{eqnarray}
\label{eq: Z_BM_fin_last} Z_t &=& \prob[\rho> t | \F_t] \nonumber\\
&=&
F_\mu(T-t, x - X_t) \indic _{\{X_t \leq x\}} +
F_{ - \mu} (T-t, X_t -x) \indic_{\{X_t > x\}}
\\
 \eqntext{\mbox{for }
t \in[0, T).}
\end{eqnarray}
In particular, $Z_0 = \prob[\rho> 0] = 1 - F_{\sign(x) \mu} (T, |x|)$.
Define also the function $h_\mu\dvtx (0, \infty) \mapsto\Real_+$ via
\begin{eqnarray*}
h_\mu(\tau)&:=& - \frac{1}{2} \biggl(\frac{\partial F_\mu}{\partial z} +
\frac{\partial F_{-\mu}}{\partial z} \biggr) (\tau, 0) \\
&= &\frac{1}{\sqrt{2
\pi\tau}} \exp \biggl(-
\frac{\mu^2 \tau}{2} \biggr) - \mu \bigl(1 - 2 \oPhi(\mu \sqrt{\tau}) \bigr)\qquad
\mbox{for } \tau\in(0, \infty).
\end{eqnarray*}
Upon differentiation, one checks that the nonnegative function $h_\mu$
is decreasing in $\tau\in(0, \infty)$.

By a straightforward generalisation of the It\^o--Tanaka formula, one
can write $Z = N - A$, where $N$ is a local martingale (with
necessarily continuous paths) and $A = \int_0^{\cdot} h_\mu(T - t)
\,\ud
\Lambda^X_t (x)$. Recalling that $\prob[\rho> 0] = 1 - F_{\sign(x)
\mu
} (T, |x|)$, it follows from \eqref{eq: real defn of K} that
%
\begin{eqnarray}
\label{eq: K_BM_fin_last} K_t = 1 - \bigl(1 - F_{\sign(x) \mu} \bigl(T, |x|\bigr)
\bigr) \exp \biggl(- \int_0^t
h_\mu(T - s) \,\ud\Lambda^X_s (x) \biggr)
\nonumber
\\[-8pt]
\\[-8pt]
\eqntext{\mbox{for } t \in[0, T).}
\end{eqnarray}
Since $L = Z / (1 - K)$, \eqref{eq: Z_BM_fin_last} and \eqref{eq:
K_BM_fin_last} give a closed-form expression for $L$.

\begin{lem} \label{lem: Q_exist_BM_fim_last}
For all $t \in[0, T)$, it holds that $\expec_\prob[L_t] = 1$.
\end{lem}

\begin{pf}
As in the proof of Lemma~\ref{lem: Q_exist_BM_fim_max}, it will be
shown that $\expec_\prob[L^*_t] < \infty$ holds for all $t \in[0, T)$.
Since $L \leq1 / (1 - K)$ and $h_\mu$ is a decreasing function, for
all $t \in[0, T)$ we obtain the inequality $L^*_t \leq(1 - F_{\sign
(x) \mu} (T, |x|))^{-1} \* \exp ( h_\mu(T - t) \Lambda^X_t (x)
 )$.
Therefore, it suffices to show that $\expec_\prob [ \exp( a
\Lambda^X_t
(x) )  ] < \infty$ holds for all $a \in\Real$ and $t \in\Real_+$.
For this, and in view of Girsanov's theorem and H\"older's inequality,
one may assume that $\mu= 0$. Then the properties of standard Brownian
motion imply that, for $\mu= 0$, the law of $\Lambda^X_t (x)$ under
$\prob$ is stochastically dominated in the first order by the law of
$\Lambda^X_t (0)$ under $\prob$. Furthermore, L\'evy's equivalence
theorem on Brownian local time and maximum of Brownian motion \cite{MR1121940},
Theorem~3.6.17, implies that the law of $\Lambda^X_t (0)$
under $\prob$ is the same as the law of $X^\up_t$ under $\prob$; the
latter is also the same as the law of $|X_t|$ under $\prob$, for which
all exponential moments are finite.
\end{pf}

By Lemma~\ref{lem: Q_exist_BM_fim_last} and the Daniell--Kolmogorov
extension theorem, there exists a probability $\qprob$ on $(\Omega,
\F)$ such that $L_t = (\ud\qprob/ \ud\prob)|_{\F_t}$ holds for all
$t \in[0, T)$. (Remark~\ref{rem: BM_fin_q_exist} becomes again relevant
at this point.) Since $L = Z / (1 - K)$, using \eqref{eq:
Z_BM_fin_last} and \eqref{eq: K_BM_fin_last} we obtain the dynamics of
$L$ as
\begin{eqnarray*}&&
\frac{\ud L_t}{L_t} = \biggl(- \frac{ (\partial F_\mu/ \partial z) (T-t, x
- X_t) }{ F_\mu(T-t, x - X_t)}\indic_{\{X_t \leq x\}}\\
&&\hspace*{38pt}{} +
\frac{
(\partial F_{ - \mu} / \partial z) (T-t, X_t - x) }{ F_{ - \mu} (T-t,
X_t - x)}\indic_{\{X_t > x\}} \biggr) \,\ud (X_t - \mu t),
\end{eqnarray*}
for $t \in[0, T)$. Then a straightforward application of Girsanov's
theorem and \eqref{eq: dF_BM_fin} imply that, under $\qprob$, the
dynamics of $X$ are given by
\begin{eqnarray}
\ud X_t = \bigl(G_\mu(T-t, x - X_t)
\indic_{\{X_t \leq x\}} - G_{ - \mu} (T-t, X_t - x)
\indic_{\{X_t > x\}} \bigr) \,\ud t + \ud W^\qprob_t\nonumber\\
\eqntext{\mbox{for } t \in[0, T),}
\end{eqnarray}
where $W^\qprob$ is a standard Brownian motion on $(\Omega, \bF,
\qprob)$ and the
function $G_\mu$ is defined in \eqref{eq: G_BM_fin}.

\begin{rem}
As was the case in Section~\ref{subsec: time_max_BM_fin}, when the
Brownian motion has zero drift the formulas simplify. In particular,
when $\mu= 0$,
\[
K_t = 1 - \biggl(1 - 2 \oPhi \biggl(\frac{|x|}{\sqrt{T}} \biggr)
\biggr) \exp \biggl(- \frac
{1}{\sqrt{2 \pi}} \int_0^t
\frac{1}{\sqrt{T - s}} \,\ud\Lambda ^X_s (x) \biggr)\qquad
\mbox{for } t \in[0, T)
\]
and, under $\qprob$, the dynamics of $X$ are given by
\begin{eqnarray}
\ud X_t = - \sign(X_t - x) \biggl(\frac{1}{\sqrt{T - t}}
\frac{\phi
(|X_t - x| / \sqrt{T - t} ) }{\oPhi(|X_t - x| / \sqrt{T - t} ) }
\biggr) \,\ud t + \ud W^\qprob_t\nonumber\\ \eqntext{\mbox{for
} t \in[0, T).}
\end{eqnarray}
\end{rem}

\section{The decomposition result of Jeulin and Yor} \label{sec: jeulin-yor}

Let $\rho$ be a $\F_\infty$-measurable random time on $(\Omega, \F, \bF)$. Let
$\bG= (\G_t)_{t \in{\Real_+}}$ be defined via
\[
\G_t = \bigl\{B \in\F_\infty| B \cap\{\rho> t\} =
B_t \cap\{\rho> t\} \mbox{ for some } B_t \in
\F_t \bigr\}, \qquad t \in{\Real_+}.
\]
It is straightforward to check that $\bG$ is a right-continuous
filtration that contains $\bF$, as well as that $\rho$ is a stopping
time on $(\Omega, \bG)$.

Whenever $X$ is a local martingale on $(\Omega, \bF, \prob)$, the Jeulin--Yor
decomposition theorem identifies the Doob--Meyer decomposition of
$X^\rho$ on $(\Omega, \bG, \prob)$. Here, we provide the statement
(Theorem~\ref
{thmm: jeulin-yor}) and a novel proof of the result of Jeulin and Yor
that uses the tools developed in this paper and does not rely on
elements of the theory of progressive filtration enlargements. The
following result, which is basically a consequence of Proposition~\ref
{prop: expect wrt Q}, provides a main ingredient of our approach. It is
useful to recall the collection $(\eta_u)_{u \in[0, 1)}$ from \eqref{eq:
good stop times}.

\begin{lem} \label{lem: loc_mart at random time}
Let $\rho$ be a $\F_\infty$-measurable random time, and $Y$ be a
process such that $\expec_\prob[Y_\rho^*] < \infty$ and $Y^{\eta
_u}$ is
local martingale on $(\Omega, \bF, {\mathbb{Q}_u})$ for all $u \in
[0, 1)$. Then $Y^\rho$ is a
martingale on $(\Omega, \bG, \prob)$.
\end{lem}

\begin{pf}
Using \eqref{eq: change of variables}, observe that $\int_{[0, 1)}
\expec_{\mathbb{Q}_u}[Y^*_{\eta_u}] \,\ud u = \expec_\prob[Y_\rho
^*] < \infty$.
Furthermore, the mapping $[0, 1)\ni u \mapsto\expec_{\mathbb
{Q}_u}[Y^*_{\eta_u}]$ is nondecreasing, as follows from consistency
of the family
$(\qprob_u)_{u \in[0, 1)}$. Therefore, $\expec_{\mathbb
{Q}_u}[Y^*_{\eta_u}] <
\infty$ for all $u \in[0, 1)$. This implies that, actually, $Y^{\eta_u}$
is a uniformly integrable martingale on $(\Omega, \bF, {\mathbb
{Q}_u})$ for all $u \in[0, 1)$.

Fix $s\in\Real_+$ and $t \in(s, \infty)$. Pick $B \in\G_{s}$ and
$B_s \in\F_s$ such that $B \cap\{\rho> s\} = B_s \cap\{\rho> s\}
$. Note that the process $Y^t \indic_{B_s \cap]\!] s, \infty[\![}$
is optional on $(\Omega, \bF)$ and $Y^t_\rho\indic_{B_s \cap\{s <
\rho\}}
= Y^\rho_t \indic_{B_s} \indic_{\{\rho> s \}}$. In view of
Proposition~\ref{prop: expect wrt Q} (with the usual trick of splitting
into positive and negative parts) and the martingale property of
$Y^{\eta_u}$ on $(\Omega, \bF, {\mathbb{Q}_u})$ for all $u \in[0,
1)$, we obtain
\begin{eqnarray*}
\expec_\prob \bigl[Y^\rho_t
\indic_{B_s} \indic_{\{\rho> s \}} \bigr] &=& \int_{[0, 1)}
\expec_{\mathbb{Q}_u} \bigl[Y^{\eta_u}_t
\indic_{B_s} \indic _{\{{\eta_u} > s\}} \bigr] \,\ud u
\\
&=& \int_{[0, 1)} \expec_{\mathbb{Q}_u} \bigl[Y^{\eta_u}_s
\indic_{B_s} \indic_{\{{\eta_u} > s\}} \bigr] \,\ud u =
\expec_\prob \bigl[Y^\rho_s \indic
_{B_s} \indic_{\{\rho> s \}} \bigr].
\end{eqnarray*}
The last equation and the fact that $Y^\rho_t \indic_B = Y^\rho_{s}
\indic_B \indic_{\{\rho\leq s \}} + Y^\rho_t \indic_{B_s}
\indic
_{\{\rho> s\}}$ imply that $\expec_\prob[Y^\rho_t \indic_B ] =
\expec_\prob[Y^\rho_{s} \indic_B ]$. Since $B \in\G_{s}$ is arbitrary,
we obtain $\expec_\prob[Y^\rho_t | \G_s] = Y^\rho_s$, which
establishes the claim.
\end{pf}

What follows is the decomposition theorem of Jeulin and Yor (see \cite
{MR519998}, as well as \cite{MR2380894} for further development), which
in particular implies that for any semi-martingale $X$ on $(\Omega,
\bF, \prob)$,
$X^\rho$ is a semi-martingale on $(\Omega, \bG, \prob)$.


\begin{thmm} \label{thmm: jeulin-yor}
Let $\rho$ be a $\F_\infty$-measurable random time on $(\Omega, \F, \bF, \prob)$ with
associated canonical pair $(K, L)$. Recall the processes $Z$ and $N$
from Section~\ref{subsec: canon pair}. Furthermore, let $X$ be a
process such that $X^{\eta_u}$ is a local martingale on $(\Omega, \bF, \prob)$ for
all $u \in[0, 1)$. Then:
\begin{longlist}[(1)]
\item[(1)] The set-inclusion $[\![ 0, \rho]\!] \subseteq\Gamma:=
 \bigcup_{u \in[0, 1)} [\![ 0, \eta_u ]\!]$ holds modulo $\prob$-  evanescence.
\item[(2)] The processes $\langle L, X \rangle$ and $\langle N, X \rangle
$, each being the
predictable compensator under $\prob$ of $[L, X]$ and $[N, X]$
respectively, are well defined on $\Gamma$.
\item[(3)]$\prob[\inf_{t \in{\Real_+}} L^\rho_{t-} > 0] = 1$ and
$\prob
[\inf_{t \in{\Real_+}} Z^\rho_{t-} > 0] = 1$; therefore, $\prob$-a.s.,
\[
\int_0^{\rho} \frac{1}{L_{t-}} \,\ud\Var \bigl(
\langle L, X \rangle \bigr)_t = \int_0^{\rho}
\frac{1}{Z_{t-}} \,\ud\Var \bigl(\langle N, X \rangle \bigr)_t <
\infty,
\]
where ``$\Var$'' is the operator returning the first variation of a process.
\item[(4)] The process
%
\begin{equation}
\label{eq: jeulin-yor} Y^\rho:= X^\rho- \int_0^{\rho\wedge\cdot}
\frac{1}{L_{t-}} \,\ud \langle L, X \rangle_t = X^\rho-
\int_0^{\rho\wedge\cdot} \frac
{1}{Z_{t-}} \,\ud\langle N, X
\rangle_t
\end{equation}
is a local martingale on $(\Omega, \bG, \prob)$.
\end{longlist}
\end{thmm}

\begin{rem}
Technicalities aside, intuition on the important statement (4) of
Theorem~\ref{thmm: jeulin-yor} follows from Lemma~\ref{lem: loc_mart at
random time} coupled with an application of Girsanov's theorem. Indeed,
if $X^{\eta_u}$ is a martingale on $(\Omega, \bF, \prob)$, $Y^{\eta
_u}$ (in obvious
notation) has (some kind of) the martingale property on $(\Omega, \bF, {\mathbb{Q}_u})$ in
view of Girsanov's theorem and the fact that $L_{\eta_u} = (\ud\qprob
_u / \ud\prob)|_{\F_{\eta_u}}$ for all $u \in[0, 1)$. Then $Y^\rho$
should have (some kind of) the martingale property on $(\Omega, \bG,
\prob)$, as
follows from Lemma~\ref{lem: loc_mart at random time}.

The idea of proving the Jeulin--Yor decomposition theorem via
Girsanov's theorem has also been used by Jeulin and Yor
\cite{MR884713}, Chapter III, page~172. However, Girsanov's theorem there is applied on the
product space $\Omega\times\Real_+$ equipped with the predictable
sigma-algebra. The approach here is more transparent, as we are dealing
with probabilities on $(\Omega, \F, \bF)$.
\end{rem}

\begin{pf*}{Proof of Theorem~\ref{thmm: jeulin-yor}}
Since $\prob[\rho\leq\eta_a] = \int_{[0, 1)} \qprob_u [\eta_u
\leq\eta_a] \,\ud u \geq a$ holds for all $a \in[0, 1)$ by Proposition~\ref
{prop: expect wrt Q}, it follows that $\lim_{a \uparrow1} \prob[\rho
\leq\eta_a] = 1$. Therefore, statement (1) is established.

Fix $u \in[0, 1)$. As $L^{\eta_u}$ is locally bounded (see Lemma~\ref
{lem: loc bdd of L}) and $X^{\eta_u}$ is locally integrable (being a
local martingale) on $(\Omega, \bF, \prob)$, it follows that $\Var
([L, X])^{\eta
_u}$ is locally integrable on $(\Omega, \bF, \prob)$. By \eqref{eq:
defn of L} and
$Z = L (1-K)$, $\Var([N, X])^{\eta_u} = (1 - K_-) \cdot\Var
([L, X])^{\eta_u} \leq\Var([L, X])^{\eta_u}$ implies that
$\Var
([N, X])^{\eta_u}$ is also locally integrable on $(\Omega, \bF,
\prob)$. Since
this holds for all $u \in[0, 1)$, $\langle L, X \rangle$ and
$\langle N, X \rangle$ are
well defined on $\Gamma$, which establishes statement (2).

By Proposition~\ref{prop: num prof of L} $\prob[L_\rho> 0] = 1$;
since $L$ is a nonnegative local martingale on $(\Omega, \bF, \prob
)$, we obtain
$\prob[\inf_{t \in{\Real_+}} L^\rho_{t-} > 0] = 1$. Then $\prob
[\inf_{t \in{\Real_+}} Z^\rho_{t-} > 0] = 1$ follows from $\prob
[\inf_{t \in{\Real_+}} L^\rho_{t-} > 0] = 1$, coupled with $\prob
[\sup_{t \in{\Real_+}} K^\rho_{t-} < 1] = \prob[K_{\rho-} < 1] =
1$ (see
Proposition~\ref{prop: unif domination}) and the relationship $Z = L (1
- K)$, holding up to $\prob$-evanescence. This shows the validity of
statement (3).

We proceed to the proof of statement (4). Since $[\![ 0, \rho]\!]
\subseteq\Gamma$ holds modulo $\prob$-evanescence, standard
localisation arguments imply the existence of a nondecreasing sequence
$(\tau_n)_{n \in\Natural}$ of stopping times on $(\Omega, \bF)$
and a $(0,
\infty)$-valued nondecreasing sequence $(C_n)_{n \in\Natural}$ such
that all
the following conditions are met: $\tau_n \leq\eta_{1 - 1/n}$ for all
$n \in\Natural$; $\lim_{n \to\infty}\prob[\rho\leq\tau_n] =
1$; $\lim_{n \to\infty}C_n =
\infty$;
$\prob [ \inf_{t \in{\Real_+}} L^{\tau_n}_{t-} \geq C_n^{-1/2}
] =
1$ for all $n \in\Natural$; $\prob [ [L, L]_{\tau_n} \leq C_n
 ] = 1$ for
all $n \in\Natural$; $\expec_\prob[X_{\tau_n}^*] < \infty$ for
all $n \in\Natural$. [In
particular, the last condition implies that $X^{\tau_n}$ is a
martingale on $(\Omega, \bF, \prob)$ for all $n \in\Natural$.]

Suppose we can show that $Y^{\rho\wedge\tau_n}$ is a local martingale
on $(\Omega, \bG, \prob)$ for all $n \in\Natural$. Then, setting
$\tau'_n:= \tau_n
\indic
_{\{\rho> \tau_n\}}+ \infty\indic_{\{\rho\leq\tau_n\}}$, we
have that $(\tau'_n)_{n \in\Natural}$ is a nondecreasing\vspace*{1pt} sequence
of stopping
times on $(\Omega, \bG)$ such that $\prob[\lim_{n \to\infty}\tau
'_n = \infty] = 1$
and $Y^{\rho\wedge\tau'_n} = Y^{\rho\wedge\tau_n}$ is a local
martingale on $(\Omega, \bG, \prob)$ for all $n \in\Natural$; it
will then follow that
$Y^{\rho}$ is a local martingale on $(\Omega, \bG, \prob)$.
Therefore, it suffices
to show that $Y^{\rho\wedge\tau_n}$ is a local martingale on
$(\Omega, \bG, \prob)
$ for all $n \in\Natural$.

We estimate $\Var([L, X])_{\tau_n} \leq[L, L]^{1/2}_{\tau_n}
[X,X]^{1/2}_{\tau_n} \leq C_n^{- 1/2} [X, X]^{1/2}_{\tau_n}$.
Using \eqref{eq: jeulin-yor} and the fact that $\inf_{t \in{\Real_+}}
L^{\tau
_n}_{t-} \geq C^{1/2}_n$, we obtain $Y^*_{\rho\wedge\tau_n} \leq
X^*_{\tau_n} +\break C_n [X, X]^{1/2}_{\tau_n}$. In view of the Davis
inequality, $\expec_\prob[X_{\tau_n}^*] < \infty$ implies\break $\expec
_\prob
[ [X, X]^{1/2}_{\tau_n}  ] < \infty$; therefore, $\expec_\prob
[Y^*_{\rho\wedge\tau_n}] < \infty$. Furthermore, $Y^{ \tau_n
\wedge
\eta_u}$ is a local martingale on $(\Omega, \bF, {\mathbb{Q}_u})$
for all $u \in[0, 1)$.
Indeed, given that, $\prob$-a.s., $\int_0^{{\tau_n} \wedge\eta_u}
( 1
/ L_{t-}) \,\ud\Var(\langle L, X \rangle)_t < \infty$, this follows
in a
straightforward way from Girsanov's theorem. Then $Y^{ \rho\wedge
{\tau_n}}$ is a martingale on $(\Omega, \bG, \prob)$, as follows
from Lemma~\ref
{lem: loc_mart at random time}.
\end{pf*}

\section*{Acknowledgment}
The author would like to thank Monique Jeanblanc for valuable input,
as well as Dan Ren for a careful reading of the manuscript and many
constructive comments.

%
%


\printaddresses
\end{document}